\documentclass[11pt,reqno]{amsart}
\usepackage{amsaddr}

\usepackage{setup}
\usepackage{microtype}

\usepackage{comment}
\usepackage[normalem]{ulem}

\newcommand{\E}{\mathbb{E}}

\newcommand{\N}{\mathbb{N}}
\renewcommand{\P}{\mathbb{P}}

\usepackage{bbm}
\newcommand{\1}{\mathbbm{1}}

\newcommand{\Cs}{\mathcal{C}}

\newcommand{\F}{\mathcal{F}}

\renewcommand{\L}{\mathcal{L}}

\newcommand{\U}{\mathcal{U}}

\newcommand{\ad}{\text{ and }}

\newcommand{\fa}{\forall}

\newcommand{\ceq}{\coloneqq} 

\newcommand{\e}{\varepsilon}

\newcommand{\es}{\emptyset}

\DeclarePairedDelimiter\normd{\lVert}{\rVert} 

\newcommand{\bpar}[1]{\left(#1\right)}
\newcommand{\bbr}[1]{\left[#1\right]}
\newcommand{\bcurly}[1]{\left\{#1\right\}}

\newcommand{\no}{\noindent}

\DeclareMathOperator*{\argmin}{arg\,min}

\newcommand{\ind}[1]{\1_{[#1]}} 

\renewcommand{\mathbb}{\mathbf}
\renewcommand{\mathbbm}{\mathbf}
\newcommand{\cT}{\mathcal{T}}

\newcommand{\MLE}{\mathrm{MLE}}
\newcommand{\Geo}{\mathrm{Geo}}
\newcommand{\Ber}{\mathrm{Ber}}
\newcommand{\Exp}{\mathrm{Exp}}
\newcommand{\Unif}{\mathrm{Unif}}
\newcommand{\T}{\mathbf{T}}

\newcommand{\dist}{\mathrm{dist}}
\renewcommand{\es}{\varnothing}

\def\draft{1}
\newcommand{\mnote}[1]{\ifnum\draft=1\textcolor{red}{[\textbf{Madhu:} #1]}\fi}
\newcommand{\cnote}[1]{\ifnum\draft=1\textcolor{violet}{[\textbf{Cassandra:} #1]}\fi}
\newcommand{\anote}[1]{\ifnum\draft=1\textcolor{teal}{[\textbf{Anna:} #1]}\fi}
\newcommand{\enote}[1]{\ifnum\draft=1\textcolor{blue}{[\textbf{Elchanan:} #1]}\fi}

\title{Finding the root in random nearest neighbor trees}
\author[A.~Brandenberger, C.~Marcussen, E.~Mossel, M.~Sudan]{\small Anna Brandenberger$^\circ$, Cassandra Marcussen$^\dagger$, Elchanan Mossel$^\circ$, Madhu Sudan$^\dagger$}
\address{$^\circ$Department of Mathematics, MIT \\ $^\dagger$School of Engineering and Applied Sciences, Harvard}
\email{abrande@mit.edu, cmarcussen@g.harvard.edu, elmos@mit.edu, madhu@cs.harvard.edu}

\begin{document}

\begin{abstract}
    We study the inference of network archaeology in growing random geometric graphs. We consider the root finding problem for a random nearest neighbor tree in dimension $d \in \mathbb{N}$, generated by sequentially embedding vertices uniformly at random in the $d$-dimensional torus and connecting each new vertex to the nearest existing vertex. More precisely, given an error parameter $\varepsilon > 0$ and the unlabeled tree, we want to efficiently find a small set of candidate vertices, such that the root is included in this set with probability at least $1 - \varepsilon$. We call such a candidate set a $\textit{confidence set}$. We define several variations of the root finding problem in geometric settings -- embedded, metric, and graph root finding -- which differ based on the nature of the type of metric information provided in addition to the graph structure (torus embedding, edge lengths, or no additional information, respectively).

    We show that there exist efficient root finding algorithms for embedded and metric root finding. For embedded root finding, we derive upper and lower bounds (uniformly bounded in $n$) on the size of the confidence set: the upper bound is subpolynomial in $1/\varepsilon$ and stems from an explicit efficient algorithm, and the information-theoretic lower bound is polylogarithmic in $1/\varepsilon$. In particular, in $d=1$, we obtain matching upper and lower bounds for a confidence set of size $\Theta\left(\frac{\log(1/\varepsilon)}{\log \log(1/\varepsilon)} \right)$.
\end{abstract}

\maketitle

\section{Introduction}
Given the structure of a network in which vertices join sequentially over time, one may be interested in how to use the network's structure to determine properties of the history of the process. For example, is it possible for the network's structure to reveal which vertex was added first, which vertices were added early, or which vertices were added late? The line of research concerned with the relationship between the network's structure and reconstructing its history is referred to as \textit{network archaeology} \cite{navlakha2011network}. Prior work has focused on network archaeology for combinatorial models of randomly growing graphs such as the uniform and preferential attachment models \cite{bubeck2017finding, contat2024eve}.

We initiate the study of network archaeology in geometric settings, studying the problem of finding the root of a tree that evolved according to the random nearest neighbor model. First, we formalize several variations of the root finding problem in geometric settings, with the variations differing based on the nature of the information available to the algorithm. We then focus on algorithms that have access to both the graph-theoretic and geometric structure of the network, with the goal of understanding how to best utilize both the combinatorial and geometric properties of the network's structure to find the root vertex.

The random nearest neighbor model -- sometimes referred to as the online nearest neighbor model -- is one of the simplest geometric growth models for trees that can be studied. This model, first studied by Steele \cite{steele1989cost}, considers a process in which vertices are sequentially placed independently and uniformly in an underlying metric space and connected to the geometrically closest vertex added previously. While the model can be defined on a general metric space, for concreteness we focus on the underlying metric space being the $d$-dimensional torus of unit volume.

\subsection*{Random nearest neighbor model}
Fix $d, n \in \N$. The random nearest neighbor process on the $d$-dimensional torus $\T^d$ of unit volume is defined as follows. For each $i \in [n]$, a vertex is embedded uniformly on $\T^d$ independently of previous vertices; we refer to vertices by their index. For $i \geq 2$, we connect the new vertex $i$ to the closest vertex among $[i-1]$ in Euclidean torus distance, breaking ties arbitrarily (no ties occur almost surely). The Euclidean torus distance between $x = (x_1, x_2, \dots, x_d) \in \T^d$ and $y = (y_1, y_2, \dots, y_d) \in \T^d$ is defined as
\begin{equation}\label{equation:torus-euclidean-distance}
    d(x, y) = \Big(\sum_{i = 1}^d \left(\min\{ |x_i - y_i|,  1 - |x_i - y_i|\} \right)^2 \Big)^{1/2}.
\end{equation}
A tree $T_n$ with $n$ vertices generated according to this process is called a random nearest neighbor tree.
We denote the random nearest neighbor process and tree as \textit{$d$-NN process} and \textit{$d$-NN tree} respectively. 

\subsection*{Network archaeology and root finding} Initiated by the seminal paper \cite{bubeck2017finding}, the root finding question is as follows: given the unlabeled tree at some large time $n$, how can we recover the initial vertex 1? Clearly, since vertices 1 and 2 are exchangeable, this cannot be done with probability greater than $1/2$. Instead, the question is then: how large of a set must be returned to ensure that the root vertex (1) is included in this set with high probability? This question was first studied by \cite{bubeck2017finding} for the uniform and preferential attachment tree models, and has since been the subject of extensive study \cite{ crane2021inference, briend2023archaeology, contat2024eve, banerjee2022root, banerjee2023degree, brandenberger2022root}.

More precisely, in \textit{non-geometric settings}, a root finding algorithm is given as input a target error parameter $\e \in (0,1)$ and a randomly grown tree of size $n$ with its labels removed, and returns a subset of vertices $H(\e,n)$ containing the root with probability $1-\e$. 

There are three natural ways to extend this question to recursive trees generated from a geometric process, which differ based on the amount of geometric information given to the algorithm. We call these \textit{embedded} root finding, \textit{metric} root finding and \textit{graph} root finding. 

\begin{defn} \label{defn:root-finding} A root finding algorithm for a random tree $T_n$ generated via a geometric process on $\T^d$ is given inputs $\e \in (0,1)$, $d \in \N$, and one of the following:
\begin{itemize}
    \item (Embedded root finding) the unlabeled tree $T_n^\circ$ embedded in $\T^d$;
    \item (Metric root finding) the adjacency matrix of the unlabeled tree $T_n^\circ$, with corresponding edge lengths;
    \item (Graph root finding) only the adjacency matrix of $T_n^\circ$.
\end{itemize}
It outputs a subset of vertices $H_{\T^d}(\e, T_n^{\circ})$ of size independent of $n$ (only depends on $\e$), which contains the root with probability $1 - \e$, i.e., 
\begin{equation}\label{eq:root-finding-condition}
  \liminf_{n \to \infty} \P_{T_n}\left\{ 1 \in H_{\T^d}(\e, T_n^{\circ})\right\} \geq 1 - \e.
\end{equation}
We refer to $H_{\T^d}(\e, T_n^{\circ})$ as $H(\e, n)$ throughout the paper when the underlying metric space $\T^d$ and input unlabeled tree $T_n^{\circ}$ on $n$ vertices are clear from the context.
\end{defn}
The first question to ask is then: \textit{does such a root finding algorithm exist}? For some models of random trees, it is impossible to output a confidence set $H(\e,n)$ whose size is independent of $n$ \cite{brandenberger2022root}. Second, if such a root finding algorithm exists, from an information-theoretic standpoint we also want to ask: \textit{what confidence set size is necessary}? We want to find both upper and lower bounds on the confidence set size $|H(\e,n)|$ returned by the algorithm.

In this paper, we focus on embedded and metric root finding. We show that there \textit{do} exist root finding algorithms whose confidence set size $|H(\e, n)|$ only depends on $\e$ and is uniformly bounded with respect to $n$, whose runtime is polynomial in $n$ and $1/\e$. In particular, a very simple metric root finding algorithm returns a confidence set of size $2^d/\e$; see Section~\ref{section:simple-algorithm}. Focusing on embedded root finding, we derive sharper results: both upper and lower bounds on the confidence set size. Our positive results are given by efficient algorithms, and our lower bounds on the confidence set size are information-theoretic.

For $d = 1$, we prove matching upper and lower bounds on the confidence set size and provide an efficient algorithm achieving this bound.

\begin{thm}\label{thm:1d-upper-and-lower}
There exist $c_1, c_2 > 0$ such that the following holds for the 1-NN model for all $n \in \mathbb{N}$ and all sufficiently small $\e > 0$. 
There exists a $O(n^2 + \log^2(1/\e))$-time embedded root finding algorithm that returns $H(\e, n)$ of size satisfying $|H(\e, n)|\log |H(\e, n)| \leq c_1 \log(1/\e)$. Additionally, it is information-theoretically impossible to provide a confidence set whose size satisfies $|H(\e, n)| \log |H(\e,n)| < c_2 \log(1/\e)$.
\end{thm}
We can equivalently write the confidence set size as $|H(\e, n)| \asymp \log(1/\e) / W(\log(1/\e)) $, where $W$ is the Lambert $W$ function. This is $|H(\e, n)| \asymp \log(1/\e) / \log \log(1/\e) $ for sufficiently small $\e > 0$.

The bound on the size of the confidence set here is exponentially smaller than the lower bounds on the size of the confidence set for the well-studied uniform and preferential attachment tree models ~\cite{bubeck2017finding,banerjee2023degree} (both subpolynomial in $1/\e$). The fact that the bound here is logarithmic in $1/\e$ indicates that the geometry of the tree retains a lot of information about the history of the process. 

Our result for the 1-NN model can be extended to the setting where the underlying space is a thin strip in two dimensions, whose thickness depends on the error parameter $\e$; see Corollary \ref{cor:2d-strip}. As we shift to the general setting of the $d$-dimensional unit torus for fixed $d \geq 2$, while we do not retain the bounds from the one-dimensional setting, we prove that it is still possible to return a set whose size is subpolynomial in $1/\e$, and matches the curent best upper bound for root finding in uniform attachment trees.

\begin{thm}\label{thm:higher-dims}
    For the $d$-NN model with $d \in \N$, $d \geq 2$, for all $n \in \mathbb{N}$ and all sufficiently small $\e > 0$, there exists a $O(n^2 + n(1/\e)^{1/(\log \log (1/\e))})$-time embedded root finding algorithm that returns a confidence set $H(\e,n)$ of size satisfying $|H(\e, n)| \leq \exp\bpar{\frac{\log(1/\e)}{\log\log(1/\e)} + d \log\log(1/\e)}$. Additionally, there exists a constant $c > 0$ such that it is information-theoretically impossible to provide a confidence set whose size satisfies $|H(\e, n)| \log \left( {|H(\e, n)|}{d^{-1/d}}\right) \leq c \log(1/\e)/d.$
\end{thm}

\subsection{Challenges and proof ideas} \label{subsec:comparisons}

Root finding algorithms in the geometric setting have access to more information than in the combinatorial settings; namely, metric and embedded root finding algorithms get not only the tree that is formed but also information about edge lengths (metric setting) and geometric embedding (embedded setting). This may seem to suggest that finding the root is inherently easier in geometric settings than in combinatorial settings, where being `easier' corresponds to the required confidence set size being smaller. However, it turns out geometric attachment processes such as the random nearest neighbor process are widely different from combinatorial processes; this leads to new challenges to their analysis. 

Previous work on root finding for combinatorial processes \cite{bubeck2017finding, banerjee2023degree, banerjee2022root, contat2024eve} has relied on constructing suitable statistics that are useful in uncovering the ages of vertices. 
These statistics are typically based on the balance of subtrees rooted at neighbors of a vertex (e.g., for uniform attachment \cite{bubeck2017finding}) or on the degrees of vertices (e.g., for preferential attachment \cite{ banerjee2023degree, contat2024eve}), which can be analyzed via martingale methods. In contrast, in the geometric setting, it is difficult to find martingales that indicate the ages of vertices, which also concentrate well. Simple graph parameters or even distances don't adequately capture the ability of a vertex to acquire new descendants. Let us highlight why existing statistics for combinatorial processes are difficult to analyze for the random nearest neighbor process. 
First, analyzing sizes of subtrees is difficult because tracking the precise probabilities of connecting to different vertices becomes intractable due to the dependence of this on the specific geometric positions of vertices. Second, degrees are not well-concentrated enough for our purposes~\cite{lichev2024new} (similarly to in uniform attachment). Due to these challenges, naturally, root finding algorithms for random nearest neighbor trees will rely on different statistics than in combinatorial settings. 

For the random nearest neighbor process, using the fact that edges tend to become shorter over time, we focus on long edges, which are expected to have arrived early on in the process. Of course, the first edge may itself not be long, but it is likely to be near a long edge. More specifically, there exist some $\ell < 1, k \in \N$ depending on $\e, d$ such that, with probability at least $1-\e$, some edge of length at least $\ell$ arrives in the first $k$ steps of the process. We can bound the number of long edges using a packing argument, so this motivates our first attempt at an algorithm and analysis: return vertices that are within a graph-theoretic radius of $k$ of a long edge. Setting $k = 1$, we achieve a confidence size of $O\left(2^d/\e\right)$. To improve over this, we try to search within a radius of $k \geq 2$. However, the degrees scale as $\Theta(\log n)$~\cite{lichev2024new}, so searching within $k$-neighborhood could produce confidence sizes that scale with $n$ with this approach. 

We overcome this challenge with a second key idea: we prune the graph to a bounded-degree induced subgraph which includes the root vertex. In particular, this subgraph excludes so-called ``covered'' vertices that lie in geometric regions corresponding to edges, which cannot be the root due to a geometric argument. Our final algorithm is as follows: return vertices \textit{in this induced subgraph of bounded degree} that are within a radius of $k$ of an edge of length at least $\ell$. This algorithm returns a confidence set whose size is sub-polynomial in $1/\e$ and uniformly bounded in terms of $\e$ and $d$.

We specify the root finding algorithms for $d=1$ and $d \geq 2$ respectively in Algorithms~\ref{alg:1d_algorithm} and \ref{alg:higher-dimensional}, and analyze them respectively in Sections~\ref{sec:1d-proof} and \ref{sec:higher-dims}. In these sections, we also study lower bounds on the confidence size. To do so, we construct a family of trees with the properties that, first, a random nearest neighbor tree is in the family with probability at least $\e$ and, second, for each tree in the family the root has a lower maximum likelihood estimator value than most other vertices in the tree.

We now make some remarks about our results. First, using the approach outlined above, in the one-dimensional setting, we prove that the size of the confidence set needed is exponentially smaller than the known lower bounds on the confidence set size for uniform and preferential attachment trees. This indicates the power of geometric information in finding the root. In the setting of $d \geq 2$, our upper bounds on the confidence set size match those currently given for uniform attachment trees \cite{bubeck2017finding}; however, since the processes evolve rather differently and the algorithms used are also different, this may not point to deeper parallels in the complexity of root finding between the two settings. Second, in the existing literature on random nearest neighbor trees, much more precise statements can often be made in the setting of $d = 1$ than in the setting of $d \geq 2$ (e.g. \cite{penrose2006limit}). The strength of our bounds (for example, as reflected in the comparisons between upper and lower bounds on the confidence set size) parallel this difference between $d = 1$ and $d \geq 2$.

\subsection*{A note on the underlying geometric space.} We focus on the setting where the underlying geometric space is the unit torus $\T^d$ in $d$ dimensions, in order to avoid issues with boundary conditions. 
We believe that the results and proofs extend to other ``nice" $d$ dimensional sets such as the $d$-dimensional sphere.  

\subsection{Related work} \label{subsec:related_work}
The problem of finding the root of a randomly growing graph process is a fundamental question in the realm of \textit{network archaeology}, which is an area of research focused on understanding how well the history of a network can be reconstructed given access to the structure of a network at a later time. Root finding is well-studied for a wide range of combinatorial models of randomly growing graphs, such as uniform attachment trees \cite{bubeck2017finding, crane2021inference} and uniform attachment DAGs \cite{briend2023archaeology}, preferential attachment trees \cite{bubeck2017finding, contat2024eve}, and more general families of graph processes \cite{banerjee2022root, banerjee2023degree}. These algorithms often return sets of vertices whose degrees are large or whose position in the graph (with respect to sizes of subtrees rooted at their neighbors) is ``balanced.'' Lichev and Mitsche \cite{lichev2024new} raise the question of if root finding is possible for random nearest neighbor trees. In terms of how root finding can impact network archaeology more broadly, understanding how to find which vertex arrived first in the process can lead to broader insights into how to detect times of arrivals of vertices and understand which vertices arrived early or late in the process. For example, in the setting of uniform attachment processes, an algorithm for finding the root in \cite{bubeck2017finding} was later adapted to the setting of estimating the entire history of the process \cite{briend2024estimating}.

Real-world networks are often modeled by processes involving geometric components \cite{flaxman2006geometric, jacob2013spatial}. Therefore, it is both interesting and relevant to understand how the geometry of the processes can impact and facilitate the discovery of the root given a snapshot of the network at a later time. The random nearest neighbor model is a useful starting point due its simplicity to state and also because it appears as a limiting case of geometric models of real-world networks (more specifically, of geometric preferential attachment \cite{jordan2015phase}). Network archaeology for geometric network models is, to our knowledge, unstudied and it remains an interesting open direction to understand if or how our approaches extend beyond the random nearest neighbor tree setting.

There is a long line of research on online nearest neighbor graphs over $(0, 1)^d$; see \cite{penrose2006limit, wade2009asymptotic,penrose2005multivariate,wade2007explicit,trauthwein2022quantitative}. These papers primarily focus on understanding the asymptotic behavior of sums of power-weighted edge lengths (meaning, consider all of the edges in the tree, raise their lengths to some power, and sum this). Additionally, Aldous \cite{aldous2018random} studies properties of the Voronoi regions in the two-dimensional case for a related geometric randomly growing graph model. A more recent paper \cite{lichev2024new} studies combinatorial properties of the random nearest neighbor tree on $\T^d$. Many geometric and combinatorial properties of random nearest neighbor trees are still unstudied or unresolved; understanding such properties could lead to improved root finding guarantees.

\subsection{Notation} For any labeled tree $T$, we let $T^\circ$ denote the corresponding unlabeled tree. 

We label the vertices of a $d$-NN tree $T_n$ by their time of arrival $i \in [n]$. Often, when we study embedded root finding, we instead refer to the vertices by their geometric position $x_i \in \T^d$ (and always do so when we refer to the unlabeled $d$-NN tree $T_n^\circ$). 

For any labeled tree $T$ on $[n]$, for each $i \leq n$, we denote by $T[i]$ the subtree induced by vertices $[i]$. We denote its edge set by $E(T)$; in $T_n$, we may refer to edges via either the vertex labels $e = (i,j)$ or positions $e = (x_i, x_j)$. For two vertices $x, y \in \T^d$, we denote by $\dist(x,y)$ their distance in the torus metric, which corresponds to the length of the edge $(x,y)$ if it exists. 

We call an embedded labeled tree $T$ a ``valid nearest neighbor tree" if for each $1 < i \leq |T|$, there exists $j < i$ such that $(i,j) \in E(T)$ and $d(x_i,x_j) = \min_{\ell<i} d(x_i,x_\ell)$. Let the set of valid embedded labeled trees be $\cT$. Note that each $T \in \cT$ can be uniquely identified by its ordered set of positions $\{x_i\}_{i \in |T|}$. Then, an embedded unlabeled tree $T^\circ$ is valid if there exists some $T' \in \cT$ such that $(T')^\circ = T^\circ$.
Given a fixed valid $T \in \cT$ with vertices $\{x_i\}_{i \in |T|}$, a permutation $\sigma : [|T|] \to [|T|]$ is called ``feasible with respect to $T$'' if the embedded tree $T^\sigma$ formed by vertices $\{x_{\sigma(i)}\}_{i\in|T|}$ satisfies $(T^\sigma)^\circ = T^\circ$; in particular, $T^\sigma \in \cT$. That is, this permutation of the arrival times yields the same unlabeled tree as $T$ does. 

Throughout, we use $\log$ to denote the base $e$ logarithm. We use standard asymptotic notation $O(\cdot)$, $ \Omega(\cdot)$, $\Theta(\cdot)$, $o(\cdot)$, $\omega(\cdot)$. The constants in this notation do not depend on the dimension $d$, which we keep explicit. 

We use the notation $X \overset{d}{=} Y$ to denote that random variables $X$ and $Y$ are equal in distribution. Let $C(d)$ be the volume of the unit $d$-dimensional ball; that is, $C(d) = \pi^{d/2}/\Gamma(d/2 + 1)$, where $\Gamma$ is the gamma function.

\subsection{Organization} In section~\ref{section:simple-algorithm}, we present a metric root finding algorithm that returns vertices adjacent to suitably defined long edges. In sections~\ref{sec:1d-proof}, \ref{sec:thin-strip} and \ref{sec:higher-dims}, we study embedded root finding. In these sections, respectively, we prove Theorem \ref{thm:1d-upper-and-lower}, the corresponding result for thin two-dimensional strips, and Theorem \ref{thm:higher-dims}. We conclude by proposing conjectures and open questions in section~\ref{sec:open-questions}.

\section{Warm-up: Metric root finding}\label{section:simple-algorithm}

We first construct a metric root finding algorithm that simply returns vertices adjacent to a long edge, where the definition of ``long'' depends on $\e$. When we move to the setting of \textit{embedded} root finding in sections~\ref{sec:1d-proof}, \ref{sec:thin-strip} and \ref{sec:higher-dims}, we still rely on this idea of the root being graph-theoretically ``close'' to a long edge, while reducing the required confidence set size by using other available geometric properties. Our simple metric root finding algorithm is as follows.

\begin{alg}
    Return all vertices adjacent to edges of length at least $\ell = \e^{1/d} \cdot {\Gamma^{1/d}\left( \frac{d}{2} + 1\right)}{\pi^{-1/2}}$, where $\Gamma$ is Euler's gamma function. 
    Note that $\ell$ is chosen so that the volume of the $d$-dimensional ball of radius $\ell$ is $\e$.
\end{alg} 

Additionally, we show that all vertices in the returned confidence set arrived early in the process, where \textit{early} is defined as some function of $d$ and $1/\e$, uniformly bounded with respect to $n$. This is a desirable property out of an algorithm since all of the vertices returned in $H(\e,n)$ are in some sense informative, as opposed to arbitrary. 

\begin{rem}
    We note that the algorithm presented above is local, in the sense that every vertex can tell if it should be returned by looking within a fixed radius of itself for a long edge (in the case of this algorithm, the radius is one). This locality indicates that it is feasible to run the algorithm efficiently in a distributed setting.
\end{rem}

\begin{lem}
\label{lem:count-long-edges}
    This algorithm satisfies the property that, for all $\e > 0$ sufficiently small and all $n \in \N$, $\P \{ 1 \in H(\e,n)\} \geq 1 - \e \ad |H(\e,n)| \leq {2^{d}}/{\e}$. The algorithm runs in $O(n)$ time. Furthermore, there exists $c > 0$ such that with probability at least $1 - \e$, $|H(\e,n)| \subseteq \{1, \dots, \tau\}$ for $\tau = \exp(cd)\cdot (1/\e)\log(1/\e)$. 
\end{lem}

\begin{proof}
    \textsc{Correctness.} Let $B_d(\ell)$ denote the volume of the $d$-dimensional ball of radius $\ell$. 
    The first edge added to the graph is long with probability $1 - B_d(\ell) = 1 - \e$. 
    Therefore, for all $n \in \N$, the root is in the returned set with probability at least $1 - \e$.

    \textsc{confidence set size.} We now bound from above the number of vertices adjacent to long edges. When a vertex $t$ at position $x$ is added to $T_{t-1}$, a long edge is created if and only if there exists an empty open ball of radius $\ell$ around $x$.
    This implies that, in the subgraph of $T_n$ consisting of all vertices adjacent to long edges, each vertex occupies a disjoint ball of radius $\ell/2$.
    Our confidence set size can therefore be bounded from above by
    the number of disjoint balls of radius ${\ell}/{2}$, i.e.\ of volume $B_d\left({\ell}/{2}\right)$, that can be packed into $\T^d$, which is 
    \[ 
    \frac{1}{B_d\left({\ell}/{2}\right)} = \frac{2^d}{B_d(\ell)} = \frac{2^d}{\e}.
    \]

    \textsc{Runtime.} The algorithm takes $O(n)$ time since per vertex it takes constant time to search for an adjacent edge of length at least $\ell$.

    \textsc{All vertices returned arrived early.} We now prove that all of the vertices in the confidence set returned arrived in the first $\tau$ steps of the process.

    Let us first consider a space without wrap-around; that is, $[0, 1]^d$. For a vertex $x \in [0,1]^d$, let $\mathcal{C}$ be a (minimal size) set $\mathcal{C}(x) = \{C_i(x)\}_{i=1}^{\gamma_d(x)}$ of cones of angle $\pi/4$ centered at $x$ whose union covers the entire space;
    that is, $\gamma_d(x) = |\mathcal{C}(x)|$. 

    We make use of the following two lemmas regarding the size of $\mathcal{C}(x)$ for any $x$ and the relationship between $\mathcal{C}(x)$ and the diameter of cells in a Voronoi diagram.

    \begin{lem}{\cite[Lemma~5.5]{devroye2013probabilistic}}\label{lem:number-of-cones}
         For all $x \in \mathbb{R}^d$, $\gamma_d(x) \leq 3^d$. 
    \end{lem}

    \begin{lem}{(Observation from the proof of \cite[Theorem~5.1]{devroye2017measure})}\label{lem:voronoi-diameter}
        Consider any set $\mathcal{S}$ of points in $\mathbb{R}^d$. For $x \in \mathcal{S}$, let let $D_i(x) = \min_{x_i \in \mathcal{S} : x_i \in C_i(x)} \normd{x-x_i}$. Consider the Voronoi diagram given by $\mathcal{S}$. Then, for all $x \in \mathcal{S}$, the diameter of the Voronoi cell of $x$ is bounded above by $\sqrt{2} \max_{i} D_i(x)$. 
    \end{lem}

    Given these lemmas, we prove that with high probability all vertices returned arrived in the first $\tau = \exp(cd)\cdot (1/\e)\log(1/\e)$ steps of the $d$-NN process over $\T^d$. Observe that if by time $\tau$, the diameter of the Voronoi cell of $x$ is bounded above by $\ell$, all neighbors of $x$ added after time $\tau$ have corresponding edge length at most $\ell$ and are not included in $H(\e,n)$. 
    Therefore, if there exists $\tau$ such that with probability at least $1 - \e$, the diameter of the Voronoi cell of every vertex $x \in T_{\tau}$ is at most $\ell$, then we have that with probability at least $1 - \e$, the confidence set consists only of vertices added in the first $\tau$ steps. We use a union bound argument to show that this holds for $\tau = \exp(cd)\cdot (1/\e)\log(1/\e)$. We use the two above lemmas from \cite{devroye2013probabilistic, devroye2017measure} to bound the diameter of the Voronoi cells. 

    To complete the proof, we now bound the diameters of the Voronoi cells, using a similar argument as \cite{devroye2017measure}. We want to show that for $\tau = \exp(cd) (1/\e)\log(1/\e)$, 
    \begin{equation}\label{eq:voronoi-union}
        \mathbb{P}\left\{\exists x \in [\tau],\  i \in [\gamma_d(x)] \text{ such that } D_i(x) > \ell/\sqrt{2}\right\} \leq \e ,  
    \end{equation}
    which, combined with Lemma \ref{lem:voronoi-diameter}, yields that with probability $1-\e$, $D_i(x) \leq \ell$ for each $x \in T_\tau$ and each $i \in [\gamma_d(x)]$. Under this event, for sufficiently small $\e > 0$, no boundary issues occur as the diameter is on the order of $\e^{1/d}$. 
    
    We now show \eqref{eq:voronoi-union}. 
    By definition of $D_i(x)$, the event $\{ D_i(x) > \ell/ \sqrt{2} \}$ implies that there does not exist a $y \in T_{\tau}$ such that $y \in C_i(x)$ and $\normd{x - y} \leq \ell/\sqrt 2$. Observe that the volume of any cone of radius $\ell/\sqrt{2}$ of angle $\pi/4$ is at least $C(d) \left( {\ell}/{\sqrt{2}}\right)^d {\gamma_d(x)}^{-1}$. Since vertices are distributed i.i.d., for each $x \in [\tau], i \in [\gamma_d(x)]$, we can bound 
    \[ \P\{ D_i(x) > \ell/\sqrt{2}\} \leq  \left(1 -C(d) \left( \frac{\ell}{\sqrt{2}}\right)^d \frac{1}{\gamma_d(x)}\right)^{(\tau - 1) } \leq \exp\left(-(\tau - 1) C(d) \left( \frac{\ell}{\sqrt{2}}\right)^d \frac{1}{\gamma_d(x)}\right).\]
    Then, a union bounding over all $x \in [\tau]$ and $i \in [\gamma_d(x)]$, where we have $\gamma_d(x) \leq 3^d$ from Lemma~\ref{lem:number-of-cones}, we find that there exists $c > 0$ such that \eqref{eq:voronoi-union} holds for our choice of $\tau$.
\end{proof}

\section{Tight bounds for the confidence set size in 1D}\label{sec:1d-proof}

We now turn to embedded root finding. In this section, we prove Theorem~\ref{thm:1d-upper-and-lower}. For our upper bound, we define an algorithm that, as introduced in Section~\ref{section:simple-algorithm}, uses the idea that the root should be close to a long edge. Another useful geometric property is that the root is never in an interval covered by an edge. As the torus $\T^1$ rapidly becomes covered by edges, this property helps keep the confidence set small.

For lower bounds on the confidence set size, we also use the structure of the subgraph of vertices that are not covered by any edge. The goal is to construct a family of trees that are reasonably likely to occur and for which the root is difficult to identify. We construct this family to be trees that begin with a path of many uncovered vertices.

Before describing our algorithm and proofs, we first introduce some structural results about the one-dimensional nearest neighbor tree.

\subsection{Structural aspects of the 1-NN tree}

Each edge $e = (x_1, x_2) \in E(T_n)$ edge covers a geometric interval; we denote this interval $I_e$. Also, recall that $T_n[t]$ denotes the induced subtree of vertices $\{1, \dots, t\}$. 

\begin{defn}[Uncovered vertices and edges]\label{def:uncovered}
   A vertex $x \in T_n$ is covered if there exists an edge $e \in E(T_n)$ such that $x \in I_e$. Otherwise, $x$ is uncovered.
   Similarly, we say that an edge $e = (x_1, x_2)$ is covered if either of its endpoints $x_1$ or $x_2$ is covered. The edge is otherwise uncovered. We denote the induced subgraph of uncovered vertices by $\U_n$.
\end{defn}

Note that the space covered by edges of $T_n$ in $d = 1$ forms a continuous interval, which we 
define formally as follows.

\begin{SCfigure}[.8][hbtp]
    \centering
    \includegraphics[width=.28\linewidth]{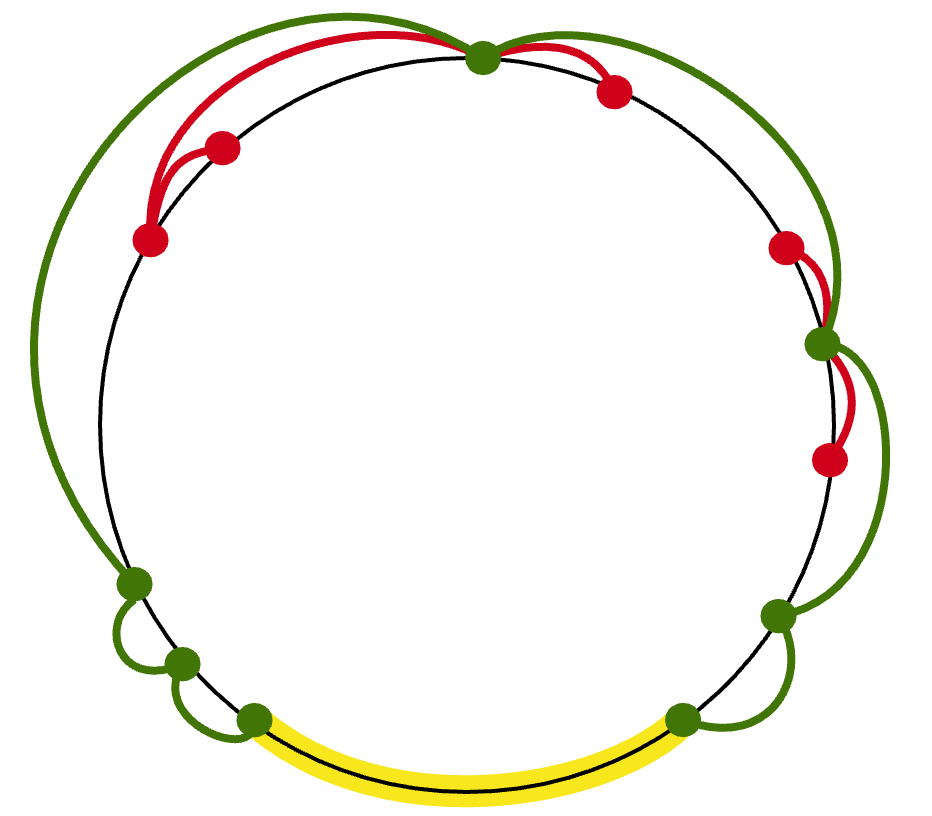}
    \hspace{.25cm}
    \caption{In this example of a 1-NN of size $n = 12$, the subtree of uncovered vertices and edges is colored green. The covered vertices and edges are colored red. The remaining uncovered space $\L_n$ is highlighted in yellow.}
    \label{fig:1d-uncovered-vertices}
\end{SCfigure}

\begin{defn}[Remaining uncovered interval]
    The remaining uncovered interval $\L_t$ at time $t$ is $\T^1 \setminus \bigcup_{e \in E(T_n[t-1])} I_e$, the continuous interval of $\T^1$ that does not intersect with any existing edge's covered interval.
\end{defn}
We now make a few observations that can be easily derived from the above definitions. 
First, note that a vertex $t \in T_n$ is uncovered if and only if it is in the remaining uncovered space at its time of arrival, i.e.\ $x_t \in \L_{t-1}$.   
Indeed, due to the nature of the $1$-NN attachment process, once a vertex is uncovered, it cannot become covered in the future. 
Also, if a vertex is uncovered, then all of its ancestors on the path to the root of the tree must also be uncovered. 

Our algorithm for root finding in $d=1$ relies on geometric and structural properties of the uncovered subgraph $\U_n$ (see Lemma \ref{lem:1d-uncovered-length}). We also require control of the size of the remaining interval $\L_t$ (see Lemma~\ref{lem:1d-uncovered-tails}).  

For $i \in \N$, we define the following random variables: let $X_i$ denote the time at which the $i$\textsuperscript{th} uncovered vertex is added, and let $L_i$ denote the length of the remaining uncovered interval after the $i$\textsuperscript{th} uncovered vertex is added. That is, $X_1 = 1, X_2 = 2, X_3 \geq 3$, and $L_1 = 1$. Letting $U_1, U_2, \dots$ be i.i.d.\ Unif$[1/2,1]$ random variables, we have 
\begin{equation}\label{eq:Ti-and-Li}
    X_{i+1} = X_i + \Geo(L_i) \ \ad \  L_{i+1} = U_i L_i, 
\end{equation}
where $X_1 = 1, L_1 = 1$, and $\Geo(p)$ is the number of independent $\Ber(p)$ trials until the first success. Note that the number of uncovered vertices is $|\U_n| = \max\{i \in \N : X_i \leq n\}$. 

\begin{lem} \label{lem:1d-uncovered-length}
    The induced subgraph of uncovered vertices $\U_n$ in $T_n$ is a path.
    In addition, if $T_n$ has $K+1$ uncovered vertices, it has expected size $X_{K+1}$ with $\E X_{K+1} = 1 + ({(\log 4)^K - 1})/({\log 4 - 1})$.  
\end{lem}

\begin{proof}
    The first statement follows since each vertex in the induced subgraph can only have one uncovered neighbor on either side.
    Next, we can compute that 
    \[
    \E \Geo(U_1) = \sum_{k=1}^\infty k \int_{1/2}^1 p(1-p)^{k-1} 2dp = \int_{1/2}^1 (2/p)dp = 2\log 2
    \]
    and that similarly, $\E \Geo(U_{j}\cdots U_1) = (2\log 2)^{j}$ for each $j \in \N$. Therefore, from \eqref{eq:Ti-and-Li}, the expected size of a 1-NN tree with $K+1$ uncovered vertices is 
    \[
    \E X_{K+1} = 2 + \E \Geo(U_1) + \E \Geo(U_1U_2) + \dots + \E \Geo(U_{1}\cdots U_{K-1}) = 1 + \sum_{j=0}^{K-1} (2\log 2)^j,
    \]
    which yields the second statement.
\end{proof}

\begin{lem}\label{lem:1d-uncovered-tails}
    For a $1$-NN tree with $K + 1$ uncovered vertices, the size of the remaining uncovered interval $L_{K+1}$ satisfies 
    \[-\log L_{K+1} \overset{d}{=} \sum_{j=1}^{K} E_j,\]
    where $E_j$ are i.i.d.\ truncated exponential $\mathrm{TEXP}(1,\log 2)$ random variables, i.e., $E_j = \Tilde{E_j}\ind{\Tilde{E_j} \leq \log 2}$ for i.i.d.\ $\Tilde{E_j} \sim \Exp(1)$.
\end{lem}
\begin{proof}
    Since each new uncovered edge removes a $\Unif(0,1/2)$ fraction of the remaining uncovered interval, the interval correspondingly shrinks by a $\Unif[1/2,1]$ factor at each step. Therefore, 
    $L_{K+1} \overset{d}{=} \prod_{j=1}^K U_j$,
    where $U_j$ are i.i.d.\ $\Unif[1/2,1]$ random variables. The statement then follows since $-\log \Unif[1/2,1] \overset{d}{=} \mathrm{TEXP}(1,\log 2)$.
\end{proof}

\subsection{Algorithm and upper-bound on the confidence set size} \label{subsection:1d-algorithm}
Consider a fixed $\e > 0$ and a constant $C > 1$ depending on $\e$. Let $k$ satisfy $k \log k = C \log(2/\e)$. Our root finding algorithm in $d = 1$ is given below.

\begin{alg}\label{alg:1d_algorithm} 
Initialize $H(\e,n)$ to be empty. While $|H(\e,n)| \leq 3k$, add uncovered vertices within graph distance $k$ of an uncovered edge of length greater than $\ell = 1/(5k)$. More precisely, 
\smallskip 

\noindent\textbf{Step 1: Find uncovered vertices.} This can be done as follows. For each unidentified vertex $x$, search among neighbors $y$ of $x$ for an edge $e = (y,z)$ such that $x \in I_e$ (as in Definition~\ref{def:uncovered}). If there is no such $e$, $x$ is marked as uncovered. Otherwise, if there exists such an $e$, then $y$ is marked as parent of $x$, and the entire subtree of $x$ is identified and marked as covered. This marks all vertices; note that a vertex may be updated from uncovered to covered.

\smallskip 
\noindent\textbf{Step 2: Building the confidence set.} For each uncovered edge of length greater than $\ell = 1/(5k)$, add all uncovered vertices within graph distance $k$ while $H(\e, n)$ is of size less than $3k$.
\end{alg}

\begin{figure}[hbtp]
    \centering
    \includegraphics[width=15cm]{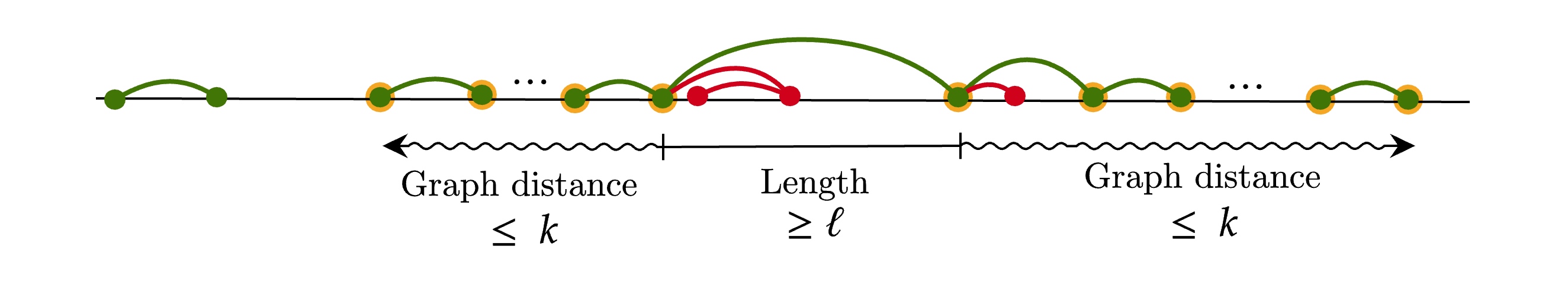}
    \caption{Our algorithm returns uncovered vertices within a graph distance of $k$ of an edge of length at least $\ell$, where $k$ and $\ell$ are defined with respect to $\e$. In this figure, the vertices highlighted in orange are added to the confidence set since they are uncovered and graph-theoretically near an edge of length at least $\ell$. In the figure, uncovered vertices and edges are green and covered vertices and edges are red.}
    \label{fig:1d-returned-set}
\end{figure}

We now argue about the correctness, confidence set size, and runtime of the algorithm to complete the proof of the upper bound in Theorem~\ref{thm:1d-upper-and-lower}.
\begin{lem} \label{lem:uncovered-algo-dimension-1} Let $d = 1$. 
For every sufficiently small $\e > 0$, there exists $C > 0$ such that for $k$ satisfying $k \log k \geq C \log(2/\e)$, the algorithm above returns a set of size $|H(\e,n)| \leq 3k$, such that with probability at least $1-\e$, the root is successfully returned in the set.
Additionally, the runtime of the algorithm is polynomial in $n$ and $\log(1/\e)$.
\end{lem}

\begin{proof} 
\textsc{Correctness.} 
We first argue that with probability at least $1 - \e/2$, the confidence set includes all vertices of distance $\leq k$ away from an edge of length $\geq \ell$. To do so, we use Lemma~\ref{lem:1d-uncovered-tails}. By Lemma \ref{lem:1d-uncovered-tails}, the remaining uncovered space $L_{k}$ after $k$ uncovered vertices are placed satisfies 
\begin{align}\label{eq:L_k-upper-tail}
\P\bcurly{L_{k} \geq \frac{1}{5k} } \nonumber &= \P \Big\{ \sum_{j=1}^{k-1} E_j \leq \log(5k) \Big\} \\ 
&\leq \binom{k-1}{(k-1)/2} \P\bcurly{ E_j \leq \frac{2\log(5k)}{k-1}}^{(k-1)/2} \leq \binom{k-1}{(k-1)/2} \bpar{\frac{4\log (5k)}{k-1}}^{(k-1)/2} \nonumber\\ 
&\leq \exp\bbr{ - \frac{k-1}{2} (\log (k-1) - \log(4\log 5) - \log \log k - \log(2e))}. \end{align}

For any sufficiently small $\e$, there exists $C > 0$ such that if $k \log k \geq C\log(2/\e)$, this probability is bounded above by $\e / 2$. 
Therefore, with probability at least $1-\e/2$, all later uncovered edges have length less than $1/(5k)$. In this case, the set of vertices of distance $k$ away from any long edge is a subset of the first $k$ uncovered vertices, plus $k$ uncovered points to either side of this path, for a total of at most $3k$ vertices. Therefore, with probability at least $1 - \e/2$, the confidence set includes all vertices of graph distance at most $k$ from an edge of length at least $\ell$.

Therefore, to argue that $\P\{1 \not\in H(\e, n)\} \geq 1 - \e$, it only remains to argue that the probability that none of the first $k$ uncovered edges has length $\geq \ell$ is at most $\e/2$. The failure probability $\P\{1 \not\in H(\e, n)\}$ can be upper bounded by the event $E$ that the first $k$ uncovered edges are all short, i.e., of length less than $\ell$. Furthermore, note that the length of the $k$th uncovered edge is uniform over $1 - \sum_{i=1}^{k-1}L_i$ where $L_i$ are the lengths of the previous $k-1$ uncovered edges. We have
\begin{equation}\label{eq:1d-prob-error}
    \P\{E\} \leq 2\ell  \frac{2\ell}{1-\ell} \cdots \frac{2\ell}{1-(k-1)\ell} < \bpar{\frac{2\ell}{1-k\ell}}^k = (5\ell/2)^k 
\end{equation}
since $1 - k\ell = 4/5$ for $\e > 0$ sufficiently small. For this failure probability to be upped bounded by $\e / 2$, we require 
\[
k \log(2k) \geq \log(2/\e),
\]
which is satisfied by our choice. Therefore, the root is successfully returned in the confidence set with probability at least $1 - \e$.

\textsc{confidence set size.} The confidence set size of the algorithm is at most $3k$ by construction.

\textsc{Runtime.} 
In the first step of the algorithm, we mark all vertices as covered or uncovered. This relies on searching within a fixed graph radius and updating labels of descendants, which takes at most $O(n^2)$ time total.
In the second step, for each uncovered edge of length greater than $\ell = 1/(5k)$, we look for uncovered vertices within graph distance $k$ of the edge. As noted in Lemma \ref{lem:1d-uncovered-length}, the induced subgraph of uncovered vertices is a path. Note also that the number of uncovered edges of length greater than $\ell = 1/(5k)$ is at most $1/\ell = 5k$. Therefore, the runtime of Step 2 is bounded above by $\log^2(1/\e)$.
Therefore, the runtime of the algorithm is $O(n^2 + \log^2(1/\e))$.
\end{proof}

\subsection{Lower bound on the confidence set size}
\label{subsec:lowerbounds-1d}

To prove a lower bound on the size of $H(\e,n)$, we look at the maximum likelihood estimator for the root on a family of trees that occurs with probability at least $\e$, as in \cite{bubeck2017finding}.  

Given a fixed embedded tree $T \in \cT$ of size $n$ with vertices $\{x_i\}_{i \in [n]}$, we consider all possible arrival orderings. 
Let $\sigma: [n] \to [n]$ be a permutation, which we think of as assigning an arrival order to the $n$ vertices. Recall that we call $\sigma$ ``feasible" if $T^\sigma$ satisfies $(T^\sigma)^\circ = T^\circ$. 

Given an unlabeled tree $T^{\circ}$ of size $n$ embedded in $\T^1$, the maximum likelihood estimator (MLE) for the root returns the vertex $i \in [n]$ maximizing the posterior probability of $i$ being the root. For the problem of determining which vertex is the root, this is the Bayes optimal answer given a uniform prior. We denote the probability of $i$ being the root as
\begin{equation}\label{MLE}
\MLE_{T^\circ}(i) = \sum_{\substack{ \sigma: [n] \to [n] \\ \sigma(i) = 1}} \P\{T_n = T^\sigma \mid T_n^\circ = T^\circ \}  =  \sum_{\substack{\sigma: [n] \to [n] \\ \text{$\sigma$ is feasible and } \sigma(i) = 1}} \P\{T_n = T^\sigma \mid T_n^{\circ} = T^\circ\}.
\end{equation}
On the right hand side, $\P\{T_n = T^\sigma | T_n^{\circ} = T^\circ\}$ is the probability that a 1-NN tree conditioned its unlabeled tree being $T^{\circ}$ has arrival times of vertices corresponding to the permutation $\sigma$. The second equality follows by noting that if $\sigma$ is not feasible, then $\{ T_n = T^\sigma \} \cap \{ T_n^\circ = T^\circ\} = \es$. 
Also, recall that a $d$-NN tree can be generated by first drawing $n$ points in $\T^d$, then drawing a random permutation $\sigma$ for the timestamps, therefore, in fact, for each feasible $\sigma$, the probability $\P\{T_n = T^\sigma \mid T_n^\circ = T^\circ \}$ is equal.

The MLE corresponds to the optimal approach for finding the root; i.e., if one wants a confidence set of size $K(\e)$, the smallest is the set $H_{\MLE}(\e,n)$ consisting of the $K(\e)$ vertices with largest posterior probability values.
To prove lower-bounds on $|H(\e,n)|$, we wish to show that, for a family of sufficiently likely tree configurations, the posterior probability of the root is small. As argued in \cite{bubeck2017finding}, it suffices to consider $H_{\MLE}(\e, K(\e) + 2)$ due to the monotonicity of the root finding property.

With this approach, we now prove the lower bound on the confidence set size in the one-dimensional setting from Theorem~\ref{thm:1d-upper-and-lower}: any root finding algorithm must return a set of size $K(\e)$ satisfying 
\begin{equation}\label{eq:1-d-lower-bound}
    (K(\e) + 1)\log(K(\e) + 5) \geq \log(2/\e).
\end{equation}

\begin{SCfigure}[.48][hbtp]
    \centering
    \includegraphics[width=.7\linewidth]{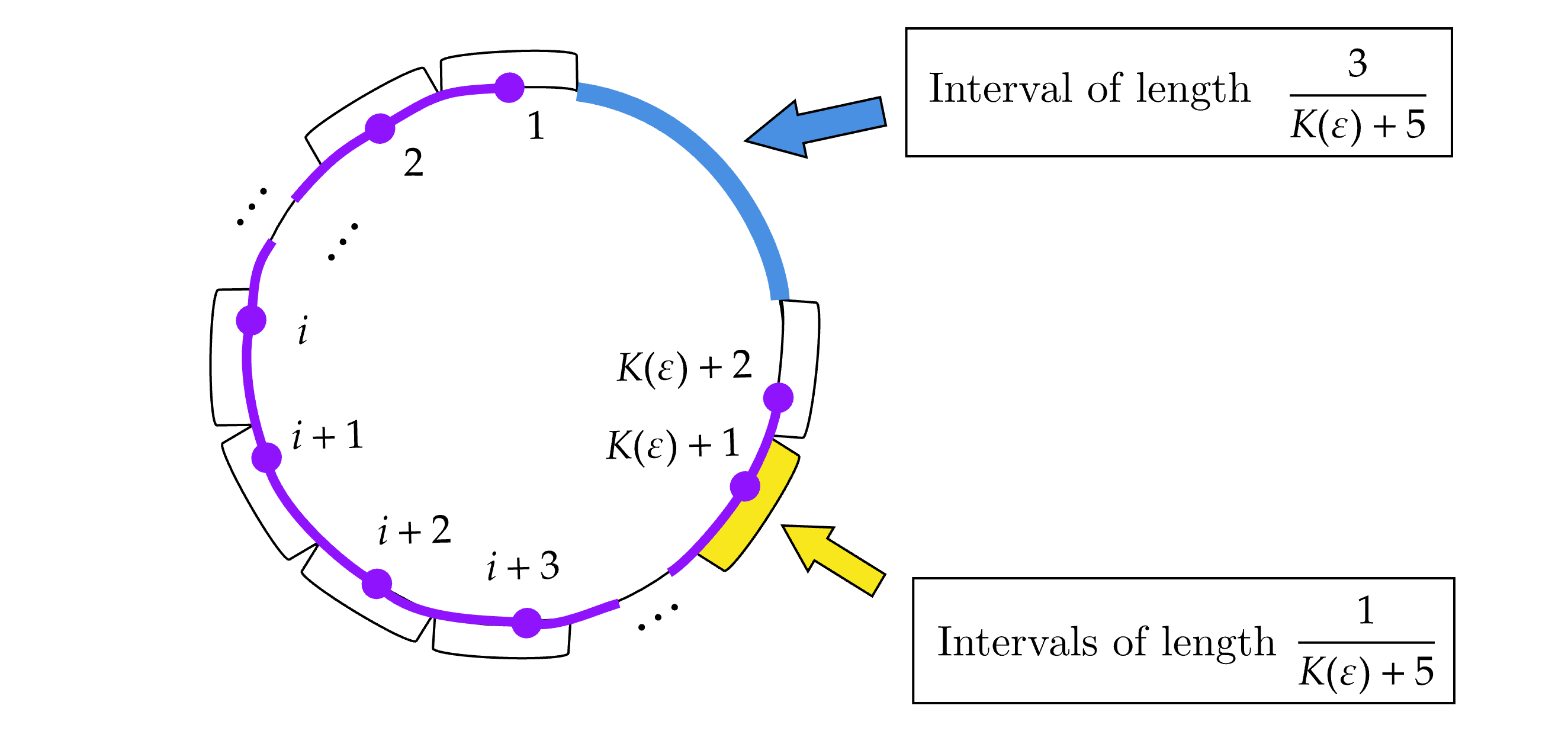}
    \caption{An example of a tree in the family $\F$ of trees considered for the lower bound on $|H(\e,n)|$ in $d = 1$. Vertices are placed sequentially in intervals of length $1/({K(\e) + 5})$ in $\T^1$. An interval of length ${3}/({K(\e) + 5})$ is left empty between the root and the last vertex added. We prove that the posterior probability of the root is lower than that of all vertices except the last one.}
    \label{fig:1d-lowerbound}
\end{SCfigure}

Consider a family $\F$ of $1$-NN trees of size $K(\e)+2$ defined as follows. 

\medskip 
\noindent \textbf{The tree configurations.} Let $T_{K(\e) + 2} \in \F$ have vertex set $\{x_t\}_{t \in [K(\e) + 2]}$, where $x_t$ is the $t$-th vertex added. First, without loss of generality, let $x_1 = 0$. 
Next, we define a set of intervals $\{I_i\}_{1 \leq i \leq K(\e)+5}$ as follows: for each $i \in [K(\e)+5]$, $I_i$ has length $1/(K(\e)+5)$ and is centered around $(i-1)/(K(\e)+5)$. 
For every $t \in \{2, 3, \dots K(\e) + 2\}$, $x_t \in I_t$. Note that $I_{K(\e) + 3}$, $I_{K(\e) + 4}$, and $I_{K(\e) + 5}$ are empty. See Figure~\ref{fig:1d-lowerbound} for an illustration of the construction.

\smallskip

We now prove the lower bound on the confidence set size in $d=1$. 
\begin{proof}[Proof of the lower bound of Theorem~\ref{thm:1d-upper-and-lower}] The proof follows from the following three claims:
\begin{enumerate}
    \item For all $T_{K(\e) + 2} \in \F$, for all $t \in \{1, 2, \dots, K(\e) + 1\}$, $x_{t + 1}$ connects to $x_t$. \label{1d-lower-bound-connections}
    \item A random $1$-NN tree is in $\F$ with probability at least $\e$ when $K(\e)$ satisfies \eqref{eq:1-d-lower-bound}.

    \item For all $T_{K(\e)+2} \in \F$, $\MLE_{T_{K(\e)+2}^\circ}(1) \leq \MLE_{T_{K(\e)+2}^\circ}(t)$ for every $t \in \{2, 3, \dots, K(\e) + 1\}$. \label{claim:mle-root-1d}
\end{enumerate}
Using these claims, we have that for $K(\e)$ satisfying \eqref{eq:1-d-lower-bound}, with probability at least $\e$, a random $1$-NN tree is in $\F$. Then, the root has a lower posterior probability than at least $K(\e)$ other vertices: any algorithm returning $K(\e)$ vertices will therefore fail to return the root. This completes the proof. We now prove the above claims. 

    \smallskip
    For (i), by construction, for any $t \leq K(\e) + 1$, $\min_{t' \leq t} \dist(x_{t+1}, x_{t'}) = \min ( \dist(x_{t+1}, x_t), \dist(x_{t+1}, x_1) )$ since $x_{t+1}$ is geometrically separated from all $1 < t' < t$ by either $t$ or $1$. We have $\dist(x_{t+1}, x_t) \leq {2}/({K(\e) + 5})$, while $\dist(x_{t+1}, x_1) \geq {3}/({K(\e) + 5})$ due to the empty intervals, so $x_{t+1}$ connects to $x_t$.

    \smallskip 
    For (ii), after placing $x_1$ arbitrarily, each of the remaining vertices is required to fall in its own adjacent interval of length $1/(K(\e)+5)$. Additionally, there are two choices of directions in which to sequentially place the future vertices. Therefore, a random 1-NN tree is in $\mathcal{F}$ with probability $2 \left(K(\e)+5\right)^{-(K(\e) + 1)}$. 

    \smallskip For (iii),
    for each $T_{K(\e) + 2} \in \F$, consider the unlabeled tree $T_{K(\e) + 2}^{\circ}$. By claim~\ref{1d-lower-bound-connections}, the only feasible permutation $\sigma$ such that $\sigma(1) = 1$ is the identity permutation where each $\sigma(t) = t$ (because, otherwise, feasibility is violated). Similarly, by claim~\ref{1d-lower-bound-connections}, for any $t \in \{2, 3, \dots, K(\e)+2\}$, the feasible permutations $\sigma$ such that $\sigma(i) = 1$ are the permutations satisfying $\sigma(t) < \sigma(t + 1) < \sigma(t + 2) < \dots < \sigma(K(\e) + 2)$ and $\sigma(t) < \sigma(t - 1) < \sigma(t - 2) < \dots < \sigma(1)$. Note that for each $i \in \{2, \dots, K(\e) + 1\}$, there exist at least two such permutations. 
    Since all feasible permutations are equally likely given $T_{K(\e)+2}^\circ$, which follows from the definition of the $1$-NN model, there exist at least $K(\e)$ vertices with higher posterior probability values than the root. 
\end{proof}

\section{Embedded root finding in a thin 2D strip}\label{sec:thin-strip}

We now consider the setting of random nearest neighbor trees evolving on a thin two-dimensional strip with a height of $h_\e = \text{poly}(\e)$. This highlights generalizations of the one-dimensional setting as well as potential limitations we encounter in moving to the higher-dimensional setting. 
The following result shows that the bounds from the $d=1$ setting transfer over to the setting where the space is a thin two-dimensional strip.

\begin{cor}\label{cor:2d-strip}
    The following holds for $\e > 0$ sufficiently small. Consider the $2$-NN model on a thin torus $[0,1] \times [0, h_\e]$. If $h_\e = O(\e^{5})$, there exists a $O(n^2 + \log^2(1/\e))$-time embedded root finding algorithm that returns a set of size $|H(\e,n)| = O\left(\frac{\log(1/\e)}{\log\log(1/\e)} \right)$ that includes the root with probability at least $1 - \e$. Additionally, if $h_\e = O\left({\log^{-1}(1/\e)}\right)$, it is information-theoretically impossible to provide a confidence set whose size satisfies $|H(\e,n)| = \Omega\left(\frac{\log(1/\e)}{\log\log(1/\e)} \right)$.
\end{cor}

\subsection{Upper bound on the confidence set size}

For a $2$-NN tree $T_n$ over a two-dimensional strip $[0,1]\times [0,h_\e]$, let the vertices be denoted $x_i \equiv (x_i^1, x_i^2)$ for $i \in [n]$. Let $\pi(T_n)$ be the projection of the tree onto $[0,1]$, with graph adjacencies maintained. That is, $\pi(T_n)$ has vertex set $\{x_i^1\}_{i \in [n]}$ and edge set $\{(x_i^1, x_j^1)\}_{(i,j) \in E(T_n)}$. The algorithm for root finding in the two-dimensional strip is as follows.

\begin{alg}  
Project $T_n^{\circ}$ to yield $\pi(T_n^\circ)$, and run the root finding algorithm for the one-dimensional random nearest neighbor graphs from Section~\ref{subsection:1d-algorithm} with $k$ satisfying $k \log k = 3\log(4/\e)$. 
\end{alg} 

Note that $\pi(T_n)$ may not be a valid 1-NN tree. That is, there may exist $t, t' \in [n]$ such that $(x_t^1, x_{t'}^1) \in E(\pi(T_n))$, with $t' < t$ but $i \neq \argmin_{i \in [t-1]} d(x_t^1, x_i^1)$. 
However, we argue that, with probability $1-\e$, the subgraph of uncovered vertices in $\pi(T_n)$ added in the \textit{first $\text{poly}(1/\e)$ steps} is ``valid" as a 1-NN tree, that is, it is the same as the subgraph of uncovered vertices in the 1-NN tree formed by points $\{x_i^1\}_{i \in [n]}$.

\begin{figure}[hbtp]
    \centering
    \includegraphics[width=.58\linewidth]{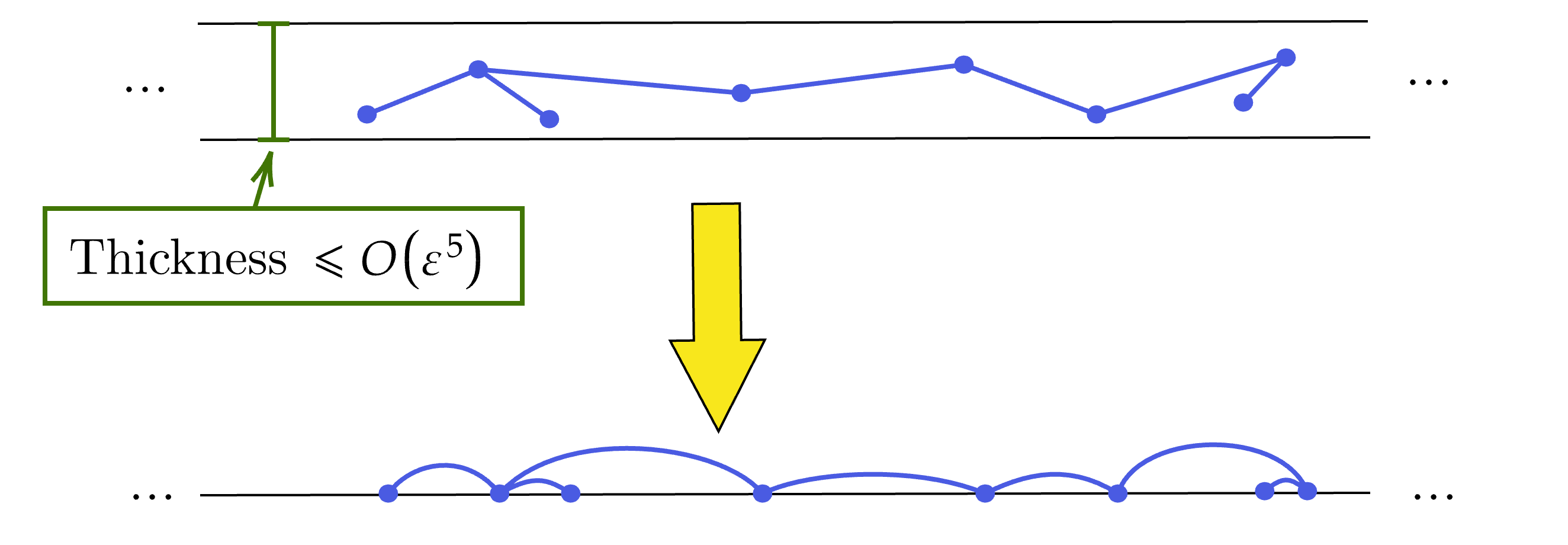}
    \caption{When the thickness of the strip is small enough, we argue that when projecting to $\T^1$, the uncovered subgraph of $\pi(T_t)$ up to time $t = \text{poly}(1/\e)$ is consistent with the uncovered graph of a $1$-NN tree with vertices in the same locations. We can therefore use algorithms from the $d=1$ setting over the thin two-dimensional torus.}
    \label{fig:2d-strip}
\end{figure}

First, we define uncovered vertices in the two-dimensional strip to be vertices uncovered in the projection to $\T^1$; note that the uncovered subgraph of $\pi(T_n)$ again has maximum degree 2. 

\begin{defn}\label{def:2d-strip-uncovered}
    For a $2$-NN tree $T_n$ over a two-dimensional strip $[0, 1] \times [0, h_\e]$, we call a vertex \textit{uncovered} if it is uncovered in the one-dimensional graph $\pi(T_n)$ in the sense of Definition \ref{def:uncovered}.
\end{defn}

Let $k$ satisfy $k \log k = 3\log(4/\e)$, as in Lemma \ref{lem:uncovered-algo-dimension-1}; that setting $C=3$ suffices for our purposes. We first argue that with probability at least $1-\e/2$, the first $k$ uncovered edges arrive before a fixed time $V_k$, and all edges added before then are long.

\begin{claim}\label{claim-4-strip}
    Let $T_n$ be a $2$-NN tree over a two-dimensional strip $[0, 1] \times [0, h_\e]$. With probability at least $1 - \e/2$, the first $k$ uncovered edges arrive before time $V_k = 3\left( {4}/{\e}\right)^{1 + 3 \log \log 4}$ and all edges $(x_\ell^1, x_{\ell'}^1) \in E(\pi(T_n))$ added before $V_k$ have length bounded from below by $d(x_\ell^1, x_{\ell'}^1) \geq L = \e/(8V_k^2)$.  
\end{claim}

\begin{proof}
    Consider $T_{V_k}$, the subtree of $T_n$ induced by the $V_k$ first vertices. As we defined a vertex to be uncovered in $T_{V_k}$ if it is uncovered in $\pi(T_{V_k})$ by Definition~\ref{def:2d-strip-uncovered}, this follows directly from Lemma~\ref{lem:1d-uncovered-length} from $d=1$ setting. We have 
    \[ \E V_{k} = 1 + \frac{(\log 4)^{k-1} - 1}{\log 4 - 1 } \leq 3 \left( {4}/{\e}\right)^{3 \log \log 4},\] 
    and applying Markov's inequality gives $\P\left\{V_{k} \geq ({4}/{\e}) \E  V_{k}\right\} \leq \e/4$. Then, for $i \in [V_k]$, the probability that the edge created by the $i$-th vertex is short ($< L$) is at most $2L \cdot (i - 1)$, and union bounding gives that the probability that any of the first $V_k$ projected edges have length less than $L$ is at most $2 L \cdot V_k^2$. By definition of $L$ and $V_k$, this probability is at most $\e/4$.
\end{proof}

\noindent \textbf{A note on the height of the strip.} We set the height $h_\e$ of the strip so that $k V_k^2 h_\e^2 / L \leq \e/4$, a bound that allows us to compare the subgraph of the first $V_k$ vertices to a $1$-NN tree with the same points. Given the definitions of $k$, $V_k$ and $L$, the inequality is satisfied for $h_\e = O(\e^{5})$.
\medskip 

We now compare $\pi(T_n)$ to a 1-NN graph.

\begin{claim}\label{claim-6-strip}
   Consider the induced subgraph of the first $V_k$ vertices $\{x_i \equiv (x_i^1, x_i^2)\}_{i \in [V_k]}$ of a $2$-NN tree $T_n$ on the thin strip, and their projection $\pi(T_n)$. On the event from Claim~\ref{claim-4-strip}, with probability at least $1 - \e/4$,
   the subgraph of the first $k$ uncovered vertices in $\pi(T_n)$ is the same as the subgraph of first $k$ uncovered vertices in the 1-NN tree generated from $\{x_i^1\}_{i \in [V_k]}$.
\end{claim}

\begin{proof}
    To prove this claim, we must show that with probability at least $1-\e/4$, for $i \in [V_k]$ that is one of the first $k$ uncovered vertices, if $(x_i^1, x_i^2)$ connects to $(x_{p_i}^1, x_{p_i}^2)$ with ${p_i} < i$ in $T_n$ then $d(x_i^1, x_{p_i}^1) = \min_{\ell<i} d(x_i^1, x_\ell^1)$, i.e., $x_i^1$ connects to $x_{p_i}^1$ in the 1-NN tree consisting of points $\{x_i\}_{i \in [V_k]}$. To do so, let $i \in [V_k]$ and suppose that $i$ connects to $p_i < i$ such that $(x_i, x_{p_i}) \in E(T_n)$, i.e., $d(x_i,x_{p_i}) = \min_{\ell < i} d(x_i,x_\ell)$. The required condition holds if for any $i \in [V_k]$ that is one of the first $k$ uncovered vertices, 
there exists no $\ell < i$, $\ell \neq p_i$ such that $d(x_i^1,x_\ell^1) < d(x_i^1,x_{p_i}^1)$.

Consider the event
\[A = \left\{\fa i \in \left\{\text{first $k$ uncovered vertices}\right\},\ \fa \ell,\ell' < i,\ |d(x_i^1,x_\ell^1) - d(x_i^1,x_{\ell'}^1) | > h_\e^2/L \right\},\] i.e., for each of the first $k$ uncovered vertices, all of its $x$-distances to other vertices are separated by at least $h_\e^2/L$. It satisfies $\P\{A^c\} \leq k V_k^2 h_\e^2 / L \leq \e / 4$ for sufficiently small $\e > 0$, where for the second inequality we used the bound on $V_k$ from Claim~\ref{claim-4-strip}. We claim that on the event $A$, if $x_i$ is closest to $x_{p_i}$ then $x_i^1$ is closest to $x_{p_i}^1$. Indeed, suppose there does exists $\ell < i,\ \ell \neq p_i$ such that $d(x_i^1,x_\ell^1) < d(x_i^1,x_{p_i}^1)$. Then, conditioned on event $A$, we have 
\begin{equation}\label{eq:bound-from-A}
    d(x_i^1, x_{p_i}^1) - d(x_i^1, x_\ell^1) > h_\e^2 / L.
\end{equation}

Now, we can use that for any $\ell, \ell' \in [V_k]^2$, 
$d(x_\ell^1, x_{\ell'}^1) \le d(x_\ell, x_{\ell'}) \le ( d(x_\ell^1, x_{\ell'}^1)^2 + h_\e^2)^{1/2}$. As we condition on the event from Claim~\ref{claim-4-strip}, we have $d(x_\ell^1, x_{\ell'}^1) > L$ and can bound 
$( d(x_\ell^1, x_{\ell'}^1)^2 + h_\e^2)^{1/2} \leq d(x_\ell^1, x_{\ell'}^1) + h_\e^2/L$, and so
\begin{equation}\label{eq:taylor-bound}
d(x_\ell^1, x_{\ell'}^1) \le d(x_\ell, x_{\ell'}) \le  d(x_\ell^1, x_{\ell'}^1) + h_\e^2/L.
\end{equation}
Using \eqref{eq:taylor-bound} for $d(x_i,x_{p_i})$ and $d(x_i,x_\ell)$ in \eqref{eq:bound-from-A}, we obtain that $d(x_i,x_\ell) < d(x_i,x_{p_i})$, contradicting $(x_i,x_{p_i}) \in E(T_n)$. 
\end{proof}

We now prove that the upper bound on the confidence set size from the $d=1$ setting holds in the setting of the thin two-dimensional strip.

\begin{proof}[Proof of the set size upper bound in Corollary \ref{cor:2d-strip}]
    \textsc{Confidence set size and correctness.} With probability at least  $1 - 3\e/4$, the events of Claim~\ref{claim-4-strip} and~\ref{claim-6-strip} both occur. On these events, for $k$ satisfying $k \log k = 3 \log(4/\e)$, the subgraph of the first $V_k$ vertices in $\pi(T_n)$ is the same as the subgraph of the first $V_k$ vertices in the 1-NN tree generated from $\{x_i^1\}_{i \in [V_k]}$. We can therefore use properties of the uncovered subgraph from the $d=1$ setting.

    As in \eqref{eq:L_k-upper-tail} when $d=1$, using Lemma~\ref{lem:1d-uncovered-tails}, the remaining uncovered space $L_{k}$ in the one-dimensional torus after $k$ uncovered vertices of $\pi(T_n)$ are placed satisfies 
    \[\P\left\{ L_{k} \geq 1/(5k) \right\} \leq \e/8,\] 
    i.e., with probability at least $1 - \e/8$, all later uncovered edges in $\pi(T_n)$ have length less than $1/(5k)$. If this occurs, the confidence set is a subset of the first $k$ uncovered vertices plus $k$ uncovered vertices to either side of this path, which is at most $3k = O\left(\frac{\log(1/\e)}{\log\log(1/\e)} \right)$ vertices.
    
    Additionally, using the same analysis as in the $d=1$ case, with probability at least $1 - \e/8$, one of the first $k$ edges in $\pi(T_n)$ is of length at least $\ell = 1/(5k)$.

    Therefore, with probability at least $1 - \e$, the confidence set consists of $O\left(\frac{\log(1/\e)}{\log\log(1/\e)} \right)$ vertices, and it includes the root.

    \textsc{Runtime.} It takes $O(n)$ time to project $T_n^{\circ}$ to $\pi(T_n^{\circ})$. After projecting the unlabeled tree, the one-dimensional root finding algorithm is run, which takes time $O(n^2 + \log^2(1/\e))$.
\end{proof}

\subsection{Lower bound on the confidence set size}
We consider a similar family $\F$ of trees as in the $d=1$ case. 
First, without loss of generality, let $x_1 = (0,0)$. Next, break the rest of strip into $2 K(\e) + 6$ evenly spaced and sized intervals $\{I_i\}_{1 \leq i \leq 2K(\e)+6}$ defined as follows: for each $i \in [2K(\e)+6]$, $I_i = \{(x,y): (i-3/2)/(2K(\e)+6) \leq x \leq (i-1/2)/(2K(\e)+6)\}$ has length $1/({2 K(\e) + 6})$ and is centered around $(i-1)/(2K(\e)+6)$. For every $t \in \{2, 3, \dots, K(\e) + 2\}$, let $x_t \in I_{2 t - 1}$. Note that the intervals $I_{2 K(\e) + 4}, I_{2 K(\e) + 5}, I_{2 K(\e) + 6}$ all do not have vertices placed in them. 

\begin{proof}[Proof of the set size lower bound in Corollary~\ref{cor:2d-strip}]
We begin by observing that for all $t \in [K(\e) + 2]$, $x_{t + 1}$ connects to $x_t$ as long as the height $h$ of the strip satisfies
\begin{equation}\label{equation:lower-bound-strip-height}
    h_\e \leq {1}/({4 K(\e) + 12}).
\end{equation}
That is, the height must be at least half of the length of one of the intervals. When this is the case, $x_{t + 1}$ is closer to $x_t$ than any other $x_{t'}$ for $t' < t$. To see this, consider the triangle formed by the vertices $(x_{t'}, x_{t}, x_{t+1})$. Let $d(x_i, x_j)$ be the Euclidean distance between any vertices $x_i$ and $x_j$. Then, $d(x_{t'}, x_{t + 1})$ must be larger than both $d(x_{t}, x_{t + 1})$ and $d(x_{t'}, x_{t})$ because it is opposite the largest angle in the triangle, because the other two angles in the triangle must be less than sixty degrees due to the height being less than half of the length of any of the intervals. Additionally, note that the distance between $x_{t + 1}$ and $x_t$ is less than ${3}/({2 (K(\e) + 5)})$ and the distance between $x_{t + 1}$ and $x_1$ is at least $6/({2 (K(\e) + 5)})$, so $x_{t + 1}$ cannot connect to any earlier vertex by ``wrapping around'' the strip. Therefore, for any tree in $\F$, for all $t > 1$, $x_{t+1}$ connects to $x_t$. 

Additionally, using the same argument as in the $d=1$ case (claim \ref{claim:mle-root-1d}, the posterior probability of the root vertex $x_1$ is lower than that of $x_t$ for all $t \in \{2, 3, \dots, K(\e) + 1\}$, therefore any confidence set returning $K(\e)$ vertices (for $K(\e)$ corresponding to the lower bound configuration) fails to return the root. This family of tree configurations occurs with probability at least $\e$ when
$$ 2 \left( 2 K(\e) + 6\right)^{-K(\e) - 1} \geq \e.$$

Therefore, when $h_\e \leq ({4 \log(2/\e) + 12})^{-1}$ and $(K(\e) + 1) \log \left( 2 K(\e) + 6\right) \leq \log(2/\e)$, any algorithm returning a set of size $K(\e)$ will fail to return the root with sufficiently high probability. 
\end{proof}

\section{Embedded root finding in higher dimensions}\label{sec:higher-dims}

We now consider $d \geq 2$, proving the bounds in Theorem~\ref{thm:higher-dims}. First, note that similarly to $d = 1$, with high probability there exists a long edge that arrives early on in the process. Therefore, long edges provide useful information for finding the root. Additionally, a notion of \textit{uncovered} vertices can be defined here. Similarly to how in $d=1$ an edge blocks off a region of space within which no vertex can be the root, in $d=2$ we also have such ``covered'' regions induced by edges. However, the definition is more involved, since the area can in some sense expand as more vertices connect to it. 

\begin{figure}[hbtp]
    \centering
    \includegraphics[width=16cm]{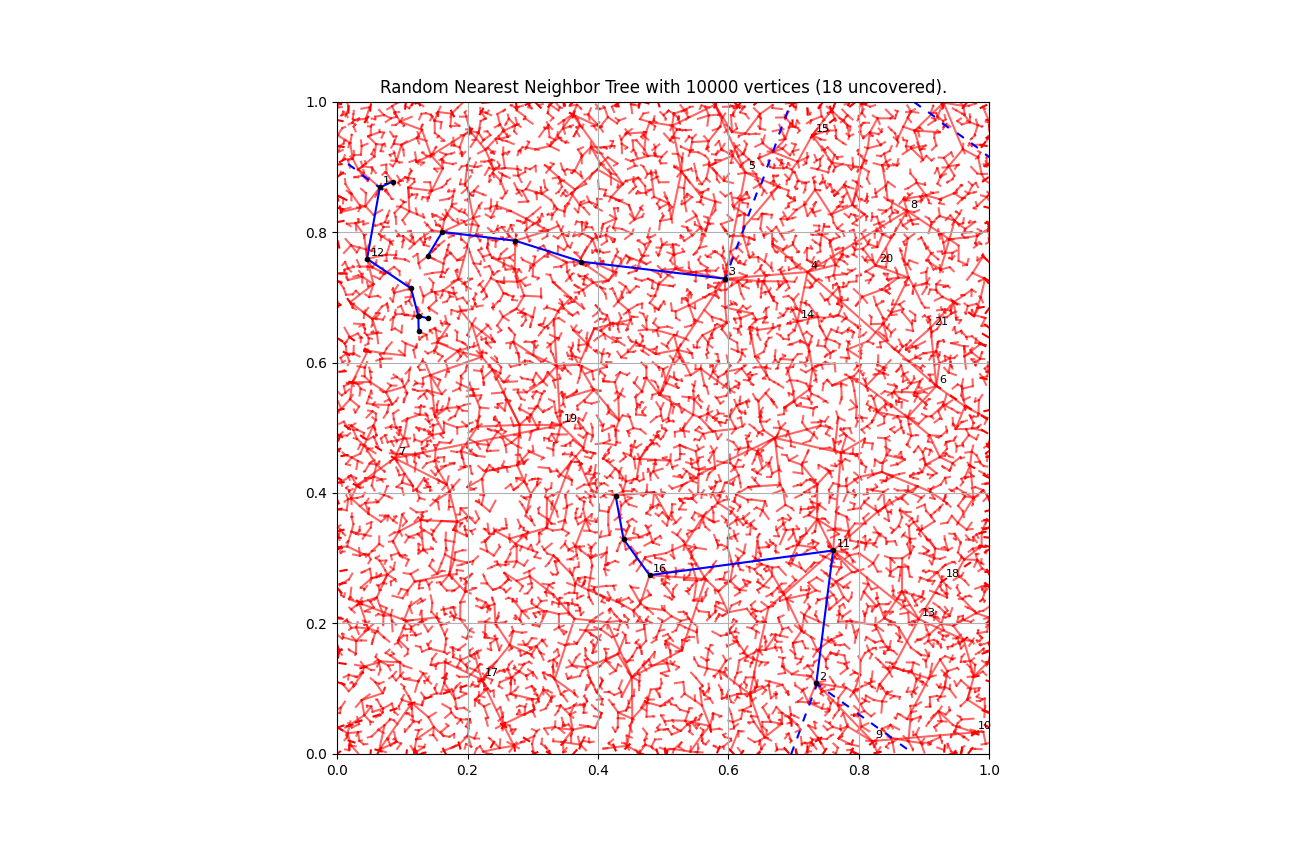}
    \caption{Simulation of a 2-NN tree of size 10,000. Uncovered edges are labeled as blue and covered edges are labeled as red. The first 20 vertices added are labeled by their time of arrival. Edges that wrap around the torus are dotted. The simulations indicate that the number of uncovered vertices in the tree grows slowly with the number of vertices in the tree.
    }
    \label{fig:simulations-uncovered-vertices}
\end{figure}

In this section, we first give a formal definition of uncovered and covered vertices in $d \geq 2$ and present a structural lemma about the maximum outdegree in the subtree of uncovered vertices. Then, we present an algorithm (Algorithm \ref{alg:higher-dimensional}) that proves the upper bound in Theorem~\ref{thm:higher-dims}. This algorithm first does a processing pass to identify uncovered vertices, then uses long edges to build the confidence set $H(\e,n)$. More precisely, the confidence set consists of uncovered vertices within a certain graph distance of a long edge. Finally, we present a lower-bound construction. Similarly to the lower-bound constructions for the one-dimensional torus and the thin two-dimensional strip, the lower-bound construction for $d \geq 2$ is a path. We can argue that every permutation of the labels yields a valid nearest neighbor tree. We then argue that the root vertex has a lower posterior probability than all but the last vertex added.

Following, for any $\ell < 1/2$, we let $C(d)$ be the constant that depends on $d$ such that the $d$-dimensional ball of radius $\ell$ has volume $C(d) \cdot \ell^d$.

\subsection{Structural information}

In $d \geq 2$, we can extend the notion of ``uncovered" and ``covered" vertices and edges.

\begin{defn}[Uncovered vertex and edge]\label{def:uncovered-general}
For each edge $(p_1, p_2)$ in the graph, if a vertex $x$ with parent $p_1$ lands in the region such that $d(x, p_1) < d(p_1, p_2) \ad d(x, p_2) < d(p_1, p_2)$, a direction can be assigned to the edge $(x,p_1)$. Particularly, orient the edge from $x$ to $p_1$ to indicate that the arrival time of $p_1$ must be earlier than that of $x$. We can similarly orient all of the vertices in the subtree of $x$. We consider $x$ and its subtree \textit{covered}. Denote by $\Cs_n = [n] \setminus \U_n$ the set of covered vertices. 

A vertex is called \textit{uncovered} if it is not a covered vertex. An edge is called \textit{uncovered} if both of its endpoints are uncovered vertices.
\end{defn}

We will consider an algorithm that first does a pass to identify covered vertices and direct their associated edges, then uses long edges to build a confidence set $H(\e,n)$. This algorithm will rely on several crucial structural components of the random nearest neighbor graph. First, the graph of uncovered vertices has bounded degree. Second, with high probability, one of the first $k \approx \log(1/\e)$ edges added is long. Third, with high probability, the earlier vertex adjacent to this edge is a graph distance at most $7 \log(k)$ away from the root.

We begin with a structural lemma that bounds the degree of any vertex in the induced subgraph of the uncovered vertices.

\begin{lem}\label{lem:uncovered-structural} 
    There exists some $a > 0$ such that the following holds for all $d \in \N$.
    The induced subgraph of the uncovered vertices $\U_n \subset [n]$ forms a tree of maximum degree $a^{d}$. 
\end{lem}
\begin{proof}

    First, note that the induced subgraph of the uncovered vertices is connected, i.e. it forms a tree. This follows from the definition of covered and uncovered: if a vertex is uncovered, there must exist a path from this vertex to the root along only uncovered edges. This applies to all uncovered vertices, yielding connectivity.

    Next, let $\deg_{\U_n}(v)$ denote the degree of $v \in \U_n$ in the induced subgraph of uncovered vertices. We seek to upper-bound the number of uncovered vertices adjacent to $v$. Consider any pair of uncovered edges $e = (u,v)$ and $e'=(u',v)$, where the arrival time of $u$ is later than that of $u'$ or $v$. For $u$ to connect to $v$ rather than $u'$, we must have $d(u, v) \leq d(u, u')$. For both $u$ and $u'$ to be uncovered, $u$ cannot lie in the ``forbidden" region of $e'$, therefore we must have $d(u', v) \leq d(u, u')$ or $d(u', v) \leq d(u, v)$.

    Therefore, the three points $v$, $u$, and $u'$ form a triangle with longest side length $d(u, u')$. The angle between $e$ and $e'$ is opposite the longest edge, and thus must be at least $\pi/3$. Therefore, the uncovered degree of any $v \in \U_n$ is upper bounded by the maximum number of vectors with pairwise angle at least $\pi/3$:
    \begin{align}
        \deg_{\mathcal{U}_n}(v) &\leq \max\left\{k : X_1, \dots, X_k \in [-1,1]^d: \frac{\langle X_i, X_j \rangle}{\normd{X_i} \cdot \normd{X_j}} \leq \cos(\pi/3) ~~ \forall i, j \in [k]\right\} \nonumber \\
        &= \max\left\{k : Z_1, \dots, Z_n \in S^{d-1} : \langle Z_i, Z_j \rangle \leq \cos(\pi/3) ~~ \forall i, j \in [k]\right\}, 
    \end{align}
    where $S^{d-1}$ is the unit ball in dimension $d$. For all $d \geq 1$, the number of unit vectors of maximum inner product $\cos(\pi/3)$ is less than $a^d$ for some constant $a > 0$, concluding the proof of the lemma.
\end{proof}

\begin{rem}
Note that in $d=1$, we have $\deg_{\U_n}(v) \leq 2$: the uncovered tree is in fact a path. In $d=2$, $\deg_{\U_n}(v) \leq 5$ almost surely (the event of any pairwise angle being exactly $\pi/3$ is measure 0).
\end{rem}

\subsection{Upper bound on the confidence set size}

We are now ready to describe the algorithm for returning a confidence set satisfying property \eqref{eq:root-finding-condition}. The algorithm consists of the following two steps.

\begin{alg}\label{alg:higher-dimensional} Let $d \geq 2$ and $0 < \e < 1/2$. \\ 
\noindent\textbf{Step 1: Assigning edge directions.} For an edge $e = (p_1, p_2)$, denote its {\sl forbidden region} to be $F(e) \ceq \{x \in [0,1]^d : d(x, p_1) < d(p_1, p_2) \ad d(x, p_2) < d(p_1, p_2)\} \subset \T^d$. For $i \in \{1,2\}$, for each neighbor $x$ of $x_i$, if $x \in F(e)$ then $x$ is covered: direct the edge $(x,x_i)$ to point towards $x_i$, and accordingly direct each other vertex $y \neq x$ in the subtree $T_x$ of $x$ (which can be determined since $(x,x_i)$ is directed); assign $y \in \Cs_n$ for each $y \in T_x$. Repeat for each remaining uncovered edge until all edges are processed, and all vertices are labeled either covered ($\in \Cs_n$) or uncovered ($\in \U_n$). 

\smallskip 
\noindent\textbf{Step 2: Building the confidence set.} Consider an edge length cutoff $\ell$ and expansion radius $k$. Let $a > 0$ be the constant from Lemma \ref{lem:uncovered-structural}. Let $\gamma = ( d \log(a) \cdot \log \log(1/\e) )^{-1}$
and set
\begin{equation}\label{eq:k-ell-settings}
    k = \gamma \log(1/\e) \ \ad \ \ell = \left({C(d) \cdot \gamma \cdot e^{1/\gamma} \cdot \log(1/\e)}\right)^{-1/d},
\end{equation}
where $C(d) $ is the volume of the unit ball in $\T^d$. Return the set of uncovered vertices that are within graph distance $k$ of either endpoint of an edge of length greater than $\ell$.
\end{alg}

\begin{SCfigure}[.7][hbtp]
    \centering
    \includegraphics[width=.53\linewidth]{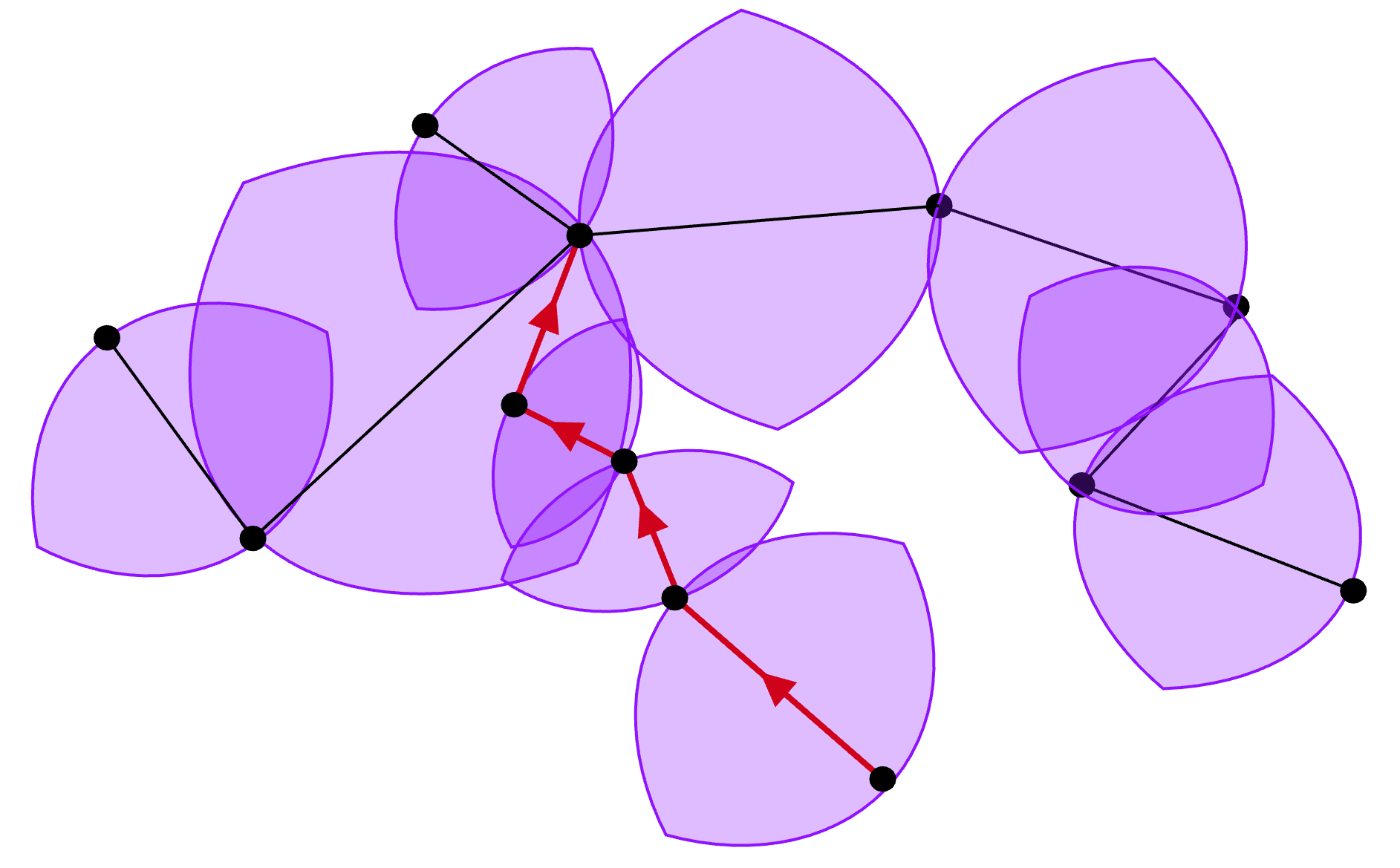}
    \caption{In this example of a $d=2$ configuration, the forbidden regions of each edge have been highlighted by the purple regions. Additionally, for vertices that are covered -- by either falling in a forbidden region of another edge or having an ancestor that is covered -- the corresponding edges have been colored red and directed accordingly. Note that the concept of ancestors in the subtree of covered vertices is well-defined because the earliest covered edge will be in the forbidden region of another edge and directed accordingly, which will induce a direction for all edges in each connected component of the covered subtree.}
    \label{fig:2d-uncovered-vertices}
\end{SCfigure}

\begin{rem}
Note that if an edge is covered, it is directed, and therefore there is a fixed direction to walk until either an uncovered edge is hit, or $k$ steps are taken. On the other hand, if we are at an uncovered edge, each endpoint can be adjacent to at most $a^d$ uncovered edges, and the exploration therefore grows in at most $a^d$ directions per uncovered step (see Lemma~\ref{lem:uncovered-structural}). 
\end{rem}

We prove that the confidence set size required by Algorithm~\ref{alg:higher-dimensional} satisfies the requirements on the confidence size specified by the upper bound in Theorem~\ref{thm:higher-dims}, and additionally that Algorithm~\ref{alg:higher-dimensional} is indeed a polynomial-time algorithm.
\begin{lem}
    First, for any fixed $d \geq 2$, there exists an $\e_d > 0$ such that for all $\e < \e_d$, the set $H(\e, n)$ returned by Algorithm \ref{alg:higher-dimensional} satisfies:
\begin{equation} \label{eq:H-bound-higher-dimension}
    |H(\e, n)| \leq \exp\bpar{\frac{\log(1/\e)}{\log\log(1/\e)} + d \log\log(1/\e)}.
\end{equation}
Second, Algorithm \ref{alg:higher-dimensional} returns a confidence set containing the root with probability at least $1 - \e$. Third, the algorithm's runtime is polynomial in $n$ and $1/\e$.
\end{lem}

\begin{proof}
    \textsc{Correctness.} First, observe that if any of the first $k$ edges are long (length $\geq \ell$), then the algorithm successfully includes the root vertex in $H(\e,n)$. The error probability is therefore bounded above by the probability that all first $k$ edges are short:
\begin{equation}
    \P\{1 \not \in H(\e,n)\} \leq \prod_{i = 1}^k \bpar{ i \cdot C(d) \ell^d } \leq (k C(d)  \ell^d)^k.
\end{equation}
With $k$ and $\ell$ set according to \eqref{eq:k-ell-settings}, we have 
\[
    \P\{1 \not \in H(\e,n)\} 
    \leq 
    \left( \frac{ C(d) \gamma \log(1/\e) }{C(d) \cdot \gamma \cdot e^{1/\gamma} \cdot \log(1/\e)}\right)^{\gamma \log(1/\e)}
    = e^{-\log(1/\e)} = \e .
\]
    \textsc{confidence set size.} Given this setting of $k$ and $\ell$, we now analyze the size $|H(\e, n)|$ of the confidence set returned by the algorithm. For each endpoint of a long edge, the algorithm adds up to $a^{dk}$ other vertices to the confidence set, because by Lemma \ref{lem:uncovered-structural}, the induced subgraph of uncovered vertices has bounded degree. Additionally, as noted in the proof of Lemma~\ref{lem:count-long-edges}, the number of long edges is at most $(2/\ell)^d$. Therefore, the size of the confidence set returned is at most
\begin{equation}
    |H(\e, n)| \leq (a)^{kd} 2^{d+1}/\ell^d = \bpar{1/\e}^{\gamma d \log(a)} \cdot 2^{d+1} \cdot C(d) \cdot \gamma \cdot e^{1/\gamma} \cdot \log(1/\e) ,
\end{equation}
and taking $\gamma = (d \log(a) \cdot \log \log (1/\e))^{-1}$ yields
\[
    |H(\e,n)| \leq \exp\bpar{\frac{\log(1/\e)}{\log\log(1/\e)}+ (d\log(a) + 1 - o(1)) \log \log (1/\e) } \frac{2^{d+1} C(d)}{d\log(a)}.
\]
There exists a universal constant $c $ such that for any fixed $d$, $\frac{2^{d+1} C(d)}{d\log(a)} \leq c$. We therefore achieve the expression in \eqref{eq:H-bound-higher-dimension}.

\textsc{Runtime.} In the first step of the algorithm, we mark all vertices as covered or uncovered. This relies on searching among neighbors and updating labels of descendants, which uses $O(n^2)$ time.

In the second step of the algorithm, the confidence set is built by having each vertex adjacent to an edge of length greater than $\ell$ return uncovered vertices within a radius $k$, where $\ell$ and $k$ are defined with respect to $\e$ and $d$ (as in \eqref{eq:k-ell-settings}). For any long edge, the number of uncovered vertices added to the confidence set as a result of being a graph distance of $\leq k$ away from this long edge is at most $a^{dk}$ which, by the setting of $k$, is $\left({1}/{\e} \right)^{\frac{1}{\log \log(1/\e)}}$. Therefore, Step 2 takes at most $n \cdot \left({1}/{\e} \right)^{\frac{1}{\log \log(1/\e)}}$ time.
\end{proof}

\subsection{Lower bound on the confidence set size}

We now consider lower bounds on the confidence set size for $d \ge 2$. As in $d=1$ setting in Section~\ref{subsec:lowerbounds-1d}, we construct a family $\F$ of tree configurations of size $K(\e)+2$ in which the root always has a lower posterior probability than $K(\e)$ other vertices in the tree. We prove the following bound, which gives the lower bound on the confidence set size in Theorem \ref{thm:higher-dims}.

\begin{lem}\label{lem:lowerbound-higherdimension}
    Any procedure for finding the root (i.e., that satisfies \eqref{eq:root-finding-condition}) for the random nearest neighbor tree model over $\T^d$ must satisfy
    \begin{equation}\label{eq:d-dimension-lower-bound}
        (K(\e) + 1) \log \left( \frac{4 K(\e) + 4}{d^{1/d}}\right) > \frac{\log(1/\e)}{d},
    \end{equation}
    regardless of its efficiency.
\end{lem}

We now describe the family of trees we construct to prove the lemma.
\smallskip

\noindent\textbf{The higher dimensional lower bound construction.} We first define cubes $\mathcal{C}(x, i, r)$ as follows. Given a point $x \in \T^d$ and $i \in [d]$, let $\mathcal{C}(x, i, r)$ be the cube consisting of points $y$ such that $x_i + r < y_i < x_i + 2r$ and for all $j \neq i \in [d]$, $x_j < y_j < x_j + r$.

Consider the following procedure for generating a tree embedded in $\T^d$. First, choose a string $\left(i_1, i_2, \dots, i_{K(\e) +1}\right) \in [d]^{K(\e) + 1}$ uniformly at random. Given this string, we now place points $(x_1, \dots, x_{K(\e+2)})$ as follows. Let $x_1 \in \T^d$ arbitrarily. Then, for $t = 2,\dots ,K(\e) + 2$, let $x_t \in \mathcal{C}(x_{t-1}, i_{t-1}, \frac{1}{4 K(\e) + 4})$.

Let $\mathcal{F}$ be the family of trees that can be generated by this procedure. We begin with the following claim about $\mathcal{F}$.

\begin{SCfigure}[.9][hbtp]
    \centering
    \hspace{.25cm}
    \includegraphics[width=.4\linewidth]{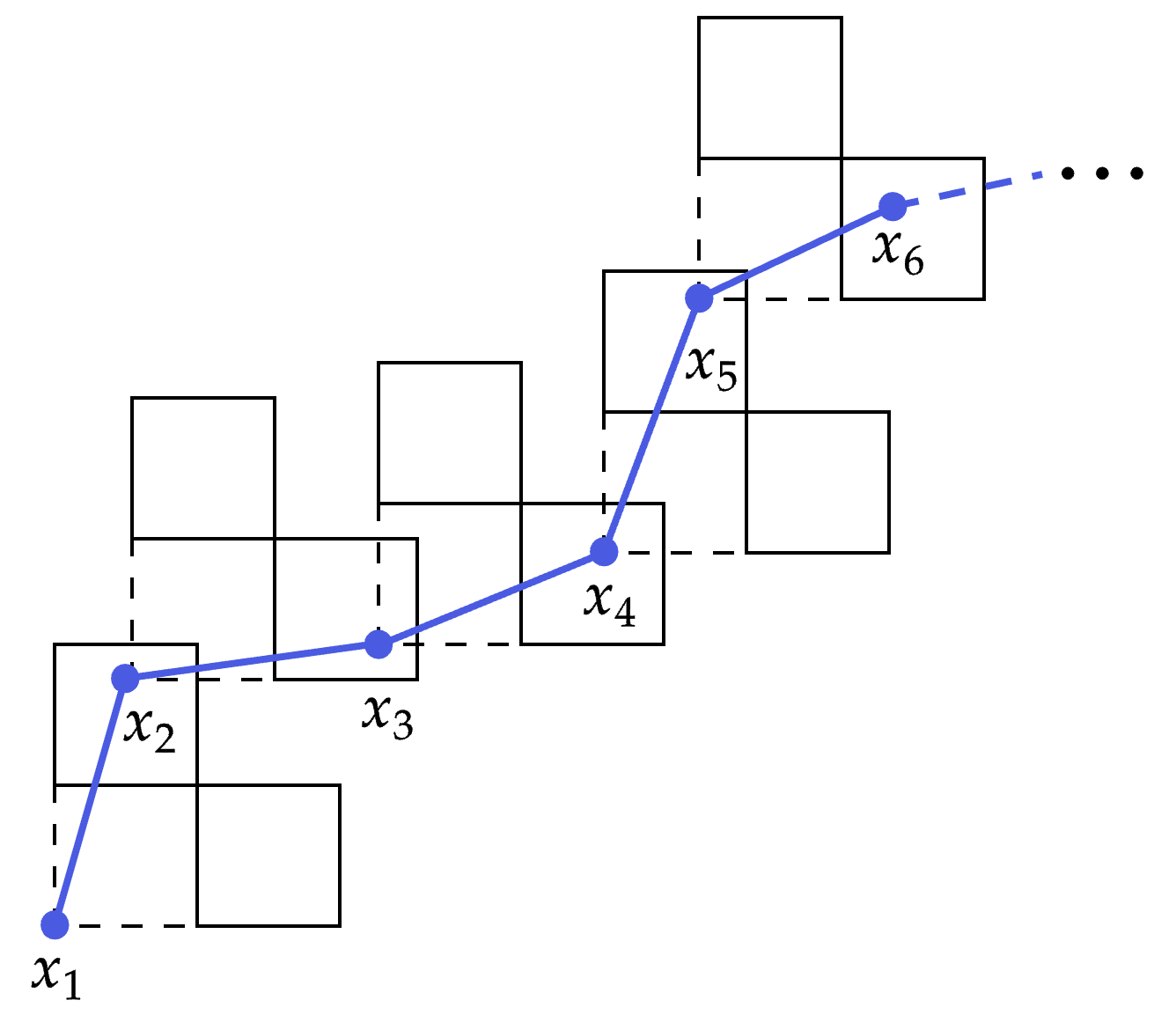}
    \caption{An example for $d = 2$ of a tree in the family of trees considered in the initial higher-dimensional lower bound. $x_1$ may fall anywhere. The remaining vertices $\{x_t\}_{t = 2}^{K(\e) + 2}$ fall sequentially in cubes $\mathcal{C}(x_{k-1}, i_{k-1}, \frac{1}{4 K(\e) + 4})$ for a randomly chosen string of directions $\left(i_1, i_2, \dots, i_{K(\e) +1}\right) \in [d]^{K(\e) + 1}$. In this visual, for each point $x_t$ added, the potential cubes $\mathcal{C}(x_t, 1, \frac{1}{4 K(\e) + 4})$ and $\mathcal{C}(x_t, 2, \frac{1}{4 K(\e) + 4})$ that the next point $x_{t+1}$ can fall into are displayed (the non-dashed squares). 
    }
    \label{fig:2d-lowerbound}
\end{SCfigure}

\begin{claim}
    For all trees in $\mathcal{F}$, for all $t \in \{1, 2, \dots, K(\e) + 1\}$, $x_{t+1}$ connects to $x_t$. 
\end{claim}

\begin{proof} 
    By construction, for any $x_{t+1}$ for $t \in \{1, 2, \dots, K(\e) + 1\}$, $x_{t+1} = x_t + w_t$, where $w_t$ is a vector with only positive entries. Specifically, the $i_t$-th entry of $w_t$ satisfies $\frac{1}{4 K(\e) + 4} < w_t(i_t) < \frac{1}{2 K(\e) + 2}$ and the rest of the entries $j$ satisfy $0 < w_t(j) < \frac{1}{4 K(\e) + 4}$. Inductively, for $t \geq 3$, $x_{t+1} = x_{t - i} + \sum_{k = t - i}^t w_{k}$.

    We want to argue that $x_{t + 1}$ is closer $x_t$ than any $x_{t - i}$. Observe that:
    $$d(x_{t + 1}, x_t) = d(0, w_t) ~~ \text{ and } ~~ d(x_{t + 1}, x_{t-i}) = d\left(0, \sum_{k = t - i}^t w_{k}\right).$$
    Since for any $t$, any $i \leq t - 1$, and any coordinate $j \in [d]$, we have $\sum_{k = t - i}^t w_{k}(j) < \frac{K(\e) + 1}{2 K(\e) + 2 } = \frac{1}{2}$, we can use the definition of the torus Euclidean distance \eqref{equation:torus-euclidean-distance} to find that 
    \[d(x_{t + 1}, x_t) = \normd{w_t}_2 ~~ \text{ and } d(x_{t + 1}, x_{t-i}) > \normd{w_t}_2,\]
    where $\normd{\cdot }_2$ denotes the $\ell_2$ norm of a vector. Therefore, by definition of the random nearest neighbor process, $x_{t+1}$ will connect to $x_t$.
\end{proof}

\begin{claim}
   The family $\mathcal{F}$ of tree configurations occurs with probability at least $\e$ when 
   \begin{equation}\label{initial-higher-dimension-lower-bound-confidence set-size}
       d^{K(\e) + 1} \cdot \left(4K(\e) + 4\right)^{-d (K(\e) + 1)} \geq \e.
   \end{equation}
\end{claim}

\begin{proof}
    In the family $\mathcal{F}$ of tree configurations considered, we impose conditions on the positions of all but the first vertex. After specifying the
    string $\left(i_1, i_2, \dots, i_{K(\e) +1}\right) \in [d]^{K(\e) + 1}$ that determines the directions we move in at each time step to construct a tree, and once $x_1$ is placed in the space (it is allowed to fall anywhere), each of the later vertices is required to fall in its own specific cube of side length $r = {1}/({4K(\e) + 4})$, which has volume $\left(4K(\e) + 4\right)^{-d} $. Therefore, once the string is chosen, the probability of such a tree configuration is $ \left(4K(\e) + 4\right)^{-d (K(\e) + 1)} $.

    Note that there are $d^{K(\e) + 1}$ different strings that can be chosen. Therefore, when \eqref{initial-higher-dimension-lower-bound-confidence set-size} holds, the family of tree configurations occurs with probability at least $\e$.
\end{proof}

\begin{claim}
    The posterior probability of the root vertex $x_1$ is lower than that of $x_t$ for $t \in \{2, 3, \dots, K(\e) + 1\}$.
\end{claim}

\begin{proof}
    This proof follows in the exact same manner as for the lower bounds over the one-dimensional torus and the thin strip. Let a feasible permutation of the arrival times be an assignment of arrival times to vertices in the unlabeled tree that is realizable by the random nearest neighbor process given the geometric positions of vertices and the edges between them.
    
    Note that, given the positions of vertices and the graph structure of a tree from the family considered, the only feasible permutation $\sigma$ such that $\sigma(1) = 1$ is the identity permutation where each $\sigma(i) = i$; else, feasibility is violated. For any other $i$, the feasible permutations such that $\sigma(i) = 1$ are the permutations satisfying $\sigma(i) < \sigma(i + 1) < \sigma(i + 2) < \dots < \sigma(K(\e) + 2)$ and $\sigma(i) < \sigma(i - 1) < \sigma(i - 2) < \dots < \sigma(1)$ (because, again, otherwise feasibility is violated). All feasible permutations are equally likely given the tree configuration, by the definition of the process. Therefore, for each tree in the family considered, there are at least $K(\e)$ other vertices that have a higher posterior probability than the root.
\end{proof}

This last claim implies that when $K(\e)$ satisfies \eqref{initial-higher-dimension-lower-bound-confidence set-size}, the optimal procedure which returns the $K(\e)$ vertices with the highest posterior probabilities fails to return the root.

\section{Conclusion and open questions} \label{sec:open-questions}

In this paper, we studied the problem of finding the root of a tree generated by the random nearest neighbor process over the $d$-dimensional unit torus $\T^d$. A root finding algorithm must return a confidence set that includes the root with probability $1 - \e$. In the setting where an algorithm is given access to the unlabeled tree and its embedding (which we call the \textit{embedded root finding} problem), when $d = 1$ we prove that the size of the confidence set is $\Theta\left( \frac{\log(1/\e)}{\log \log(1/\e)}\right)$. For fixed $d \geq 2$, the confidence set size is uniformly bounded in terms of $\e$ and $d$, and our upper bounds are sub-polynomial in $1/\e$. The difference in the bounds for the setting of $d = 1$ versus $d \geq 2$ leads to the following open question.
\medskip 

\no \textbf{Question 1.}
    For the $d$-NN model with $d \in \mathbb{N}$, does there exist a poly$(n, 1/\e)$-time root finding algorithm that returns a set of size $|H(\e, n)| \leq \log^d(1/\e)$ for all $n \in \mathbb{N}$?
\medskip

Key to our bounds is the notion of \textit{uncovered vertices}, whose induced subgraph has bounded degree and must contain the root. We believe that a strengthened understanding and analysis of the induced subgraph of uncovered vertices can be useful in obtaining the conjectured set size. We propose the following conjectures regarding uncovered vertices, which may or may not suffice for resolving the conjecture above.

\medskip
\noindent\textbf{Bounds on uncovered vertices.} Our algorithms in both the one-dimensional and higher-dimensional settings rely on uncovered vertices (see Definitions \ref{def:uncovered} and \ref{def:uncovered-general}). In the higher-dimensional setting, we use that the outdegree of any uncovered vertex in the uncovered subtree is at most $a^d$ for some constant $a > 0$. However, we believe there are more properties of uncovered vertices that can be explored and utilized. Particularly we pose the following questions.
\medskip 

\no \textbf{Question 2.}
For a tree $T_n$ and fixed dimension $d \in \N$, let $\mathcal{U}_n$ be the set of uncovered vertices (as in Definition \ref{def:uncovered-general}). Is $\E |\U_n| = O(\text{poly} \log n)$? Additionally, for random nearest neighbor trees, is $|\U_n| = O(\text{poly} \log n)$ with probability at least $1 - 1/n$?
\smallskip 

\no \textbf{Question 3.}
        More strongly, for fixed dimension $d \in \N$, consider any set $S = \{x_1, x_2, \dots, x_n\} \subset \mathbb{T}^d$. Let $\Sigma = \{\sigma\}$ be the set of all permutations $\sigma: [n] \to [n]$. For a permutation $\sigma$, let $T_n(S, \sigma)$ be the nearest neighbor tree corresponding to $\{x_{\sigma(i)}\}_{i \in [n]}$. Let $\mathcal{U}_n(T_n(S, \sigma))$ be the set of uncovered vertices in $T_n(S, \sigma)$. 
        Is it true that, for any set $S$, $\mathbb{E}_{\sigma}|\mathcal{U}_n(T_n(S, \sigma))| = O(\text{poly} \log n)$?
        
\medskip
In the new direction of network archaeology for randomly growing graph processes in geometric settings, there are many interesting open questions to be explored. We highlight several of these below.
\medskip 

\noindent\textbf{Root finding without geometric information.} The algorithms in our paper crucially utilize the geometry of the tree, particularly the lengths of edges and the notion of regions in the space that edges can ``cover'' (see Definitions~\ref{def:uncovered} and \ref{def:uncovered-general}). It is open to explore the question of root finding without geometric information; that is, \textit{graph root finding} as defined in Definition~\ref{defn:root-finding}. Algorithms would need to rely on combinatorial properties, e.g.\ centrality measures as studied in \cite{bubeck2017finding}, or other properties explored in \cite{lichev2024new}.

\smallskip
\noindent\textbf{Improving the confidence set size when the dimension is large.} The upper-bounds on the confidence set size match those for uniform attachment trees when the dimension $d$ is treated as a constant. However, as $d \to \infty$, the random nearest neighbor process behaves as a uniform attachment process because each new vertex is roughly the same distance away from each existing vertex. Removing the dependence on $d$ in high dimensions to achieve the uniform attachment bound as $d \to \infty$ remains open.

\smallskip
\noindent\textbf{Network archaeology beyond root finding.} One can study other properties of the history of the tree, such as the times of arrival of all vertices. One can also ask to return a confidence set including a vertex added late in the process.

\section{Acknowledgements}
AB was supported in part by Vannevar Bush Faculty Fellowship ONR-N00014-20-1-2826 and NSF GRFP 2141064. CM was supported in part by an NDSEG fellowship, and by NSF Award 2152413 and a Simons Investigator Award to Madhu Sudan. EM was supported in part by a Simons Investigator Award, Vannevar Bush Faculty Fellowship ONR-N00014-20-1-2826, NSF award CCF 1918421, and ARO MURI W911NF1910217. MS was supported in part by a Simons Investigator Award and NSF Award CCF 2152413. 

We thank Byron Chin, Travis Dillon, Dieter Mitsche, and Mehtaab Sawhney for helpful conversations.

\bibliographystyle{alpha}
\bibliography{bibliography}

\end{document}